\documentclass[12pt]{amsart}%
\usepackage{amsfonts}
\usepackage{amsmath}
\usepackage{amssymb}
\usepackage{graphicx}%
\providecommand{\U}[1]{\protect\rule{.1in}{.1in}}
\newtheorem{theorem}{Theorem}[section]
\theoremstyle{plain}

\newtheorem{corollary}{Corollary}[section]

\newtheorem{lemma}{Lemma}[section]

\newtheorem{proposition}{Proposition}[section]
\newtheorem{remark}{Remark}[section]

\numberwithin{equation}{section}
\begin{document}
\title[Sharp Adams type inequalities]{Sharp Adams type inequalities in Sobolev spaces $W^{m,\frac{n}{m}}\left(
 \mathbb{R}^{n}\right)  $ for arbitrary integer $m$ }
\author{Nguyen Lam}
\author{Guozhen Lu}
\address{Nguyen Lam and Guozhen Lu\\
Department of Mathematics\\
Wayne State University\\
Detroit, MI 48202, USA\\
Emails: nguyenlam@wayne.edu and gzlu@math.wayne.edu}
\thanks{Corresponding Author: G. Lu at gzlu@math.wayne.edu}
\thanks{Research is partly supported by a US NSF grant DMS0901761.}
\date{\today}
\keywords{Moser-Trudinger type inequalities, Adams type inequalities, Comparison principle, rearrangement, polyharmonic operators.}

\begin{abstract}
The main purpose of our paper is to prove sharp Adams-type inequalities in
unbounded domains of $\mathbb{R}^{n}$ for the Sobolev space $W^{m,\frac{n}{m}%
}\left(
\mathbb{R}
^{n}\right)  $ for any positive integer $m$ less than $n$. Our results
complement those of Ruf and Sani \cite{RS} where such inequalities are only
established for even integer $m$. Our inequalities are also a generalization
of the Adams-type inequalities in the special case  $n=2m=4$ proved in \cite{Y}   and stronger than those in
\cite{RS} when $n=2m$ for all positive integer $m$ by using different Sobolev norms.

\end{abstract}
\maketitle

\section{Introduction}

\bigskip Let $\Omega\subset%
\mathbb{R}
^{n}$, $n\geq2$ be a bounded domain. The Sobolev embedding theorems say that
$W_{0}^{k,p}\left(  \Omega\right)  \subset L^{q}\left(  \Omega\right)  ,~1\leq
q\leq\frac{np}{n-kp},~kp<n$ and that $W_{0}^{k,\frac{n}{k}}\left(
\Omega\right)  \subset L^{q}\left(  \Omega\right)  ,~1\leq q<\infty$. However,
we can show by easy examples that $W_{0}^{k,\frac{n}{k}}\left(  \Omega\right)
\nsubseteq L^{\infty}\left(  \Omega\right)  $. In this case, Yudovich
\cite{Yu}, Pohozaev \cite{Po} and Trudinger \cite{Tru} independently showed
that $W_{0}^{1,n}\left(  \Omega\right)  \subset L_{\varphi_{n}}\left(
\Omega\right)  $ where $L_{\varphi_{n}}\left(  \Omega\right)  $ is the Orlicz
space associated with the Young function $\varphi_{n}(t)=\exp\left(
\left\vert t\right\vert ^{n/(n-1)}\right)  -1$. In his 1971 paper \cite{Mo},
J. Moser finds the largest positive real number $\beta_{n}=n\omega
_{n-1}^{\frac{1}{n-1}}$, where $\omega_{n-1}$ is the area of the surface of
the unit $n-$ball, such that if $\Omega$ is a domain with finite $n-$measure
in Euclidean $n-$space $%
\mathbb{R}
^{n},~n\geq2,$ then there is a constant $c_{0}$ depending only on $n$ such
that
\[
\frac{1}{\left\vert \Omega\right\vert }\int_{\Omega}\exp\left(  \beta
\left\vert u\right\vert ^{\frac{n}{n-1}}\right)  dx\leq c_{0}%
\]
for any $\beta\leq\beta_{n}$, any $u\in W_{0}^{1,n}\left(  \Omega\right)  $
with $\int_{\Omega}\left\vert \nabla u\right\vert ^{n}dx\leq1$. Moreover, this
constant $\beta_{n}$ is sharp in the meaning that if $\beta>\beta_{n}$, then
the above inequality can no longer hold with some $c_{0}$ independent of $u$.
Such an inequality is nowadays known as Moser-Trudinger type inequality.

Moser's result for first order derivatives was extended to high order
derivatives by D. Adams \cite{A}. Indeed, Adams found the sharp constants for
higher order Moser's type inequality. To state Adams' result, we use the
symbol $\nabla^{m}u$, $m$ is a positive integer, to denote the $m-$th order
gradient for $u\in C^{m}$, the class of $m-$th order differentiable functions:%
\[
\nabla^{m}u=\left\{
\begin{array}
[c]{l}%
\bigtriangleup^{\frac{m}{2}}u\text{ \ \ \ \ \ \ }\mathrm{for}%
\,\,\,m\,\,\mathrm{even}\\
\nabla\bigtriangleup^{\frac{m-1}{2}}u\,\,\ \mathrm{for}\,\,\,m\,\,\mathrm{odd}%
\end{array}
\right.  .
\]
where $\nabla$ is the usual gradient operator and $\bigtriangleup$ is the
Laplacian. We use $||\nabla^{m}u||_{p}$ to denote the $L^{p}$ norm ($1\leq
p\leq\infty$) of the function $|\nabla^{m}u|$, the usual Euclidean length of
the vector $\nabla^{m}u$. We also use $W_{0}^{k,p}(\Omega)$ to denote the
Sobolev space which is a completion of $C_{0}^{\infty}(\Omega)$ under the norm
of $||u||_{L^{p}(\Omega)}+%
{\displaystyle\sum\limits_{j=1}^{k}}
||\nabla^{j}u||_{L^{p}(\Omega)}$. Then Adams proved the following:

\bigskip\textbf{Theorem A.} \textit{Let }$\Omega$\textit{ be an open and
bounded set in } $\mathbb{R}^{n}$\textit{. If } $m$ \textit{is a positive
integer less than }$n$, \textit{then there exists a constant} $C_{0}=C(n,m)>0$
\textit{such that for any} $u\in W_{0}^{m,\frac{n}{m}}(\Omega)$ \textit{and}
$||\nabla^{m}u||_{L^{\frac{n}{m}}(\Omega)}\leq1$, \textit{then}
\[
\frac{1}{|\Omega|}\int_{\Omega}\exp(\beta|u(x)|^{\frac{n}{n-m}})dx\leq C_{0}%
\]
\textit{for all} $\beta\leq\beta(n,m)$ \textit{where}
\[
\beta(n,\ m)\ =\left\{
\begin{array}
[c]{c}%
\frac{n}{w_{n-1}}\left[  \frac{\pi^{n/2}2^{m}\Gamma(\frac{m+1}{2})}%
{\Gamma(\frac{n-m+1}{2})}\right]  ^{\frac{n}{n-m}}\text{ }\mathrm{when}%
\,\,\,m\,\,\mathrm{is\,\,odd}\\
\frac{n}{w_{n-1}}\left[  \frac{\pi^{n/2}2^{m}\Gamma(\frac{m}{2})}{\Gamma
(\frac{n-m}{2})}\right]  ^{\frac{n}{n-m}}\text{ \ }\mathrm{when}%
\,\,\,m\,\,\mathrm{is\,\,even}%
\end{array}
\right.  .
\]
\textit{ Furthermore, for any $\beta>\beta(n,m)$, the integral can be made as
large as possible.}\newline

Note that $\beta(n,1)$ coincides with Moser's value of $\beta_{n}$ and
$\beta(2m,m)=2^{2m}\pi^{m}\Gamma(m+1)$ for both odd and even $m$.

The Adams inequality was extended recently by Tarsi \cite{Ta}. More precisely,
Tarsi used the Sobolev space with Navier boundary conditions $W_{N}%
^{m,\frac{n}{m}}\left(  \Omega\right)  $ which contains the Sobolev space
$W_{0}^{m,\frac{n}{m}}\left(  \Omega\right)  $\textit{ }as a closed subspace:

\bigskip\textbf{Theorem B.} \textit{Let }$n>2$\textit{ and }$\Omega\subset%
\mathbb{R}
^{n}$\textit{ be a bounded domain. Then there exists a constant }%
$C_{0}=C(n,m)>0$\textit{ such that for any }$u\in W_{N}^{m,\frac{n}{m}}%
(\Omega)$\textit{ with }$||\nabla^{m}u||_{L^{\frac{n}{m}}(\Omega)}\leq
1$\textit{ }%
\[
\frac{1}{|\Omega|}\int_{\Omega}\exp(\beta|u(x)|^{\frac{n}{n-m}})dx\leq C_{0}%
\]
\textit{for all }$\beta\leq\beta(n,m)$\textit{. Furthermore, the constant
}$\beta(n,m)$\textit{ is sharp in the sense that if }$\beta>\beta
(n,m)$\textit{ then the supremum is infinite.}\newline

The Adams inequality was also extended to compact Riemannian manifolds without
boundary by Fontana \cite{Fontana}. Also, the singular Moser-Trudinger
inequalities and the singular Adams inequalities which are the combinations of
the Hardy inequalities, Moser-Trudinger inequalities and Adams inequalities
are established in \cite{Ad2, LaLu4}.

The Moser-Trudinger's inequality and Adams inequality play an essential role
in geometric analysis and in the study of the exponential growth partial
differential equations where, roughly speaking, the nonlinearity behaves like
$e^{\alpha\left\vert u\right\vert ^{\frac{n}{n-m}}}$ as $\left\vert
u\right\vert \rightarrow\infty$. Here we mention Atkinson-Peletier \cite{AP},
Carleson-Chang \cite{CC}, Adimurthi et al. \cite{Ad1, Ad2, ASt, Ad3,  Ad5, AdY}, de
Figueiredo-Miyagaki-Ruf \cite{DeMR}, J.M. do \'{O} \cite{Do}, de Figueiredo-
do \'{O}-Ruf \cite{DeDoR},   Lam-Lu \cite{LaLu1,
LaLu3} and the references therein.

We notice that when $\Omega$ has infinite volume, the Moser-Trudinger's
inequality and Adams inequality don't make sense since the left hand side is
trivial. The sharp Moser-Trudinger type inequality for the first order
derivatives in the case $\left\vert \Omega\right\vert =+\infty$ was obtained
by B. Ruf \cite{R} in dimension two and Y.X. Li-Ruf \cite{LR} in general
dimension. In fact, such an inequality at the subcritical case was derived
earlier by Cao \cite{C} in dimension two and by Adachi and Tanaka in high
dimensions \cite{AT}. Recently, Ruf and Sani proved the Adams type inequality
for \textbf{higher derivatives of even orders} when $\Omega$ has infinite
volume. Indeed, Ruf and Sani proved the following Adams type inequality (see
\cite{RS}):

\bigskip\textbf{Theorem C. }\textit{Let }$m$\textit{ be \textbf{an even
integer} less than }$n$\textit{. There exists a constant }$C_{m,n}>0$\textit{
such that for any domain }$\Omega\subseteq%
\mathbb{R}
^{n}$%
\[
\underset{u\in W_{0}^{m,\frac{n}{m}}\left(  \Omega\right)  ,\left\Vert
u\right\Vert _{m,n}\leq1}{\sup}\int_{\Omega}\phi\left(  \beta_{0}\left(
n,m\right)  \left\vert u\right\vert ^{\frac{n}{n-m}}\right)  dx\leq C_{m,n}%
\]
\textit{ where }%
\begin{align*}
\beta_{0}\left(  n,m\right)   &  =\frac{n}{\omega_{n-1}}\left[  \frac
{\pi^{\frac{n}{2}}2^{m}\Gamma\left(  \frac{m}{2}\right)  }{\Gamma\left(
\frac{n-m}{2}\right)  }\right]  ^{\frac{n}{n-m}},\\
\phi(t) &  =e^{t}-%
{\displaystyle\sum\limits_{j=0}^{j_{\frac{n}{m}}-2}}
\frac{t^{j}}{j!}\\
j_{\frac{n}{m}} &  =\min\left\{  j\in%
\mathbb{N}
:j\geq\frac{n}{m}\right\}  \geq\frac{n}{m}.
\end{align*}%
\[
\left\Vert u\right\Vert _{m,n}=\left\Vert \left(  -\Delta+I\right)  ^{\frac
{m}{2}}u\right\Vert _{\frac{n}{m}}%
\]
\textit{This inequality is sharp in the sense that if we replace} $\beta
_{0}(n,m)$ \textit{by any } $\beta>\beta_{0}(n,m)$, \textit{then the supremum
is infinite.}

We note that the norm $||u||_{n, m}$ used in Theorem C is equivalent to the
Sobolev norm
\[
\left\Vert u\right\Vert _{W^{m,\frac{n}{m}}}=\left(  \left\Vert u\right\Vert
_{\frac{n}{m}}^{\frac{n}{m}}+%
{\displaystyle\sum\limits_{j=1}^{m}}
\left\Vert \nabla^{j}u\right\Vert _{\frac{n}{m}}^{\frac{n}{m}}\right)
^{\frac{m}{n}}.
\]
In particular, if $u\in W_{0}^{m,\frac{n}{m}}\left(  \Omega\right)  $ or $u\in
W^{m,\frac{n}{m}}\left(
\mathbb{R}
^{n}\right)  $, then $\left\Vert u\right\Vert _{W^{m,\frac{n}{m}}}\leq$
$\left\Vert u\right\Vert _{m,n}$.

The work of Ruf and Sani raised a good open question: \textbf{Does Theorem C
hold when $m$ is odd?}

One of the primary purposes of this paper is to answer the above question in
an affirmative way. This is stated as follows:

\begin{theorem}
Let $m$ \textit{be an odd integer less than }$n$: $m=2k+1,~k\in%
\mathbb{N}
$ and let $\beta(n,m)$ be as in Theorem A and the function $\phi$ be as in
Theorem C. Then there holds%
\[
\underset{u\in W^{m,\frac{n}{m}}\left(
\mathbb{R}
^{n}\right)  ,\left\Vert \nabla\left(  -\Delta+I\right)  ^{k}u\right\Vert
_{\frac{n}{m}}^{\frac{n}{m}}+\left\Vert \left(  -\Delta+I\right)
^{k}u\right\Vert _{\frac{n}{m}}^{\frac{n}{m}}\leq1}{\sup}\int_{%
\mathbb{R}
^{n}}\phi\left(  \beta\left(  n,m\right)  \left\vert u\right\vert ^{\frac
{n}{n-m}}\right)  dx<\infty.
\]

Moreover, \textit{the constant} $\beta(n,m)$ \textit{is sharp in the sense
that if we replace} $\beta(n, m)$ \textit{by any } $\beta>\beta(n, m)$,
\textit{then the supremum is infinity}.
\end{theorem}

In the special case $n=2m$ and $m$ an arbitrary positive integer, we can prove
the following stronger result which is the second main theorem of this paper:

\begin{theorem}
If $m=2k+1,~k\in%
\mathbb{N}
$, then for all $\tau>0$, there holds%
\[
\underset{u\in W^{m,2}\left(
\mathbb{R}
^{2m}\right)  ,\left\Vert \nabla\left(  -\Delta+\tau I\right)  ^{k}%
u\right\Vert _{2}^{2}+\tau\left\Vert \left(  -\Delta+\tau I\right)
^{k}u\right\Vert _{2}^{2}\leq1}{\sup}\int_{%
\mathbb{R}
^{2m}}\left(  e^{\beta(2m, m)u^{2}}-1\right)  dx<\infty.
\]

If $m=2k,~k\in%
\mathbb{N}
$, then for all $\tau>0$, there holds%
\[
\underset{u\in W^{m,2}\left(
\mathbb{R}
^{2m}\right)  ,\left\Vert \left(  -\Delta+\tau I\right)  ^{k}u\right\Vert
_{2}\leq1}{\sup}\int_{%
\mathbb{R}
^{2m}}\left(  e^{\beta(2m,m)u^{2}}-1\right)  dx<\infty.
\]
Moreover, the constant $\beta(2m,m)$ is sharp in the above inequalities in the
sense that if we replace $\beta(2m,m)$ by any $\beta>\beta(2m,m)$, then the
supremums will be infinity.
\end{theorem}

We note that for $m=2k+1$ and any $a_{0}=1,a_{2}>0,\cdots,a_{m}>0$, there is
some $\tau>0$ such that (see Lemma 2.2):
\[
\left\Vert \nabla\left(  -\Delta+\tau I\right)  ^{k}u\right\Vert _{2}^{2}%
+\tau\left\Vert \left(  -\Delta+\tau I\right)  ^{k}u\right\Vert _{2}^{2}%
\leq\sum\limits_{j=0}^{m}a_{m-j}\int_{\mathbb{R}^{n}}\left\vert \nabla
^{j}u\right\vert ^{2}dx
\]
and for $m=2k$ and any $a_{0}=1,a_{2}>0,\cdots,a_{m}>0$, there is some
$\tau>0$ such that (see Lemma 2.1):
\[
\left\Vert \left(  -\Delta+\tau I\right)  ^{k}u\right\Vert _{2}^{2}\leq
\sum\limits_{j=0}^{m}a_{m-j}\int_{\mathbb{R}^{n}}\left\vert \nabla
^{j}u\right\vert ^{2}dx.
\]
Thus, as a consequence, we will be able to establish the third main theorem of
this paper. Namely, we will replace the norm $\left\Vert \cdot\right\Vert
_{m,n}$ by $\left\Vert \cdot\right\Vert _{W^{m,\frac{n}{m}}}$ in the above
Theorem C in the case $n=2m$ for all positive integer $m$.

\begin{theorem}
Let $m\geq1$ be an integer number. For all constants $a_{0}=1,a_{1}%
,...,a_{m}>0$, there holds%
\[
\underset{u\in W^{m,2}\left(
\mathbb{R}
^{2m}\right)  ,\int_{%
\mathbb{R}
^{2m}}\left(
{\displaystyle\sum\limits_{j=0}^{m}}
a_{m-j}\left\vert \nabla^{j}u\right\vert ^{2}\right)  dx\leq1}{\sup}\int_{%
\mathbb{R}
^{2m}}\left[  \exp\left(  \beta\left(  2m,m\right)  \left\vert u\right\vert
^{2}\right)  -1\right]  dx<\infty.
\]
Furthermore this inequality is sharp, i.e., if $\beta(2m,m)$ is replaced by
any $\beta>\beta(2m,m)$, then the supremum is infinite.
\end{theorem}

In the special case $n=2m=4k=4$, the above theorem was proved by Yang in
\cite{Y}.

As a corollary of the above theorem, we have the following Adams type
inequality with the standard Sobolev norm:

\begin{theorem}
Let $m\geq1$ be an integer number. There holds%
\[
\underset{u\in W^{m,2}\left(
\mathbb{R}
^{2m}\right)  ,\left\Vert u\right\Vert _{W^{m,2}}\leq1}{\sup}\int_{%
\mathbb{R}
^{2m}}\left[  \exp\left(  \beta\left(  2m,m\right)  \left\vert u\right\vert
^{2}\right)  -1\right]  dx<\infty.
\]
Furthermore this inequality is sharp, i.e., if $\beta(2m,m)$ is replaced by
any $\beta>\beta(2m,m)$, then the supremum is infinite.
\end{theorem}

Since the fact that if $u\in W_{0}^{m,\frac{n}{m}}\left(  \Omega\right)  $ or
$u\in W^{m,\frac{n}{m}}\left(
\mathbb{R}
^{n}\right)  $, then $\left\Vert u\right\Vert _{W^{m,\frac{n}{m}}}\leq$
$\left\Vert u\right\Vert _{m,n}$, our result is stronger than the one in
\cite{RS} in the case $m$ is even$.$ Moreover, our theorems still hold when
$m$ is odd.

We organize this paper as follows: In Section 2, we provide some
preliminaries. We build an iterated comparison in Section 3 and use it to
prove the Adams type inequalities  (Theorem 1.2, Theorem 1.3 and Theorem 1.4) for the case $n=2m=4k,~k\in%
\mathbb{N}
$, namely when $m$ is \textbf{even} in Section 4. Section 5 is devoted to
proving Theorems 1.2, 1.3 and 1.4 when $n=2m=4k+2$, namely when $m$ is \textbf{odd}.
In fact, we will first prove these theorems in the special case when $n=2m=6$.
Then we will prove these theorems in the general case $n=2m=2(2k+1)$.
Finally, the Adams-type inequality when $m$ is odd in general (Theorem 1.1) is
proved in Section 6.

\section{Preliminaries}

In this section, we provide some preliminaries. For $u\in W^{m,2}\left(
\mathbb{R}
^{2m}\right)  $ with $1\leq p<\infty$, we will denote by $\nabla^{j}u$,
$j\in\left\{  1,2,...,m\right\}  $, the $j-th$ order gradient of $u$, namely
\[
\nabla^{j}u=\left\{
\begin{array}
[c]{l}%
\bigtriangleup^{\frac{j}{2}}u\text{ \ \ \ \ \ \ }\mathrm{for}%
\,\,\,j\,\,\mathrm{even}\\
\nabla\bigtriangleup^{\frac{j-1}{2}}u\,\,\ \mathrm{for}\,\,\,j\,\,\mathrm{odd}%
\end{array}
\right.  .
\]

For $m=2k,~k\in%
\mathbb{N}
$, $\tau>0$, we have the following observations:
\begin{equation}
\left(  -\Delta+\tau I\right)  ^{k}u=%
{\displaystyle\sum\limits_{i=0}^{k}}
\left(  -1\right)  ^{k-i}\binom{k}{i}\tau^{i}\Delta^{k-i}u \label{2.01}%
\end{equation}
where
\[
\binom{k}{j}=\frac{k!}{j!(k-j)!}.
\]
Thus
\begin{align*}
\int_{%
\mathbb{R}
^{2m}}\left\vert \left(  -\Delta+\tau I\right)  ^{k}u\right\vert ^{2}dx  &
=\int_{%
\mathbb{R}
^{2m}}\left\vert
{\displaystyle\sum\limits_{i=0}^{k}}
\left(  -1\right)  ^{k-i}\binom{k}{i}\tau^{i}\Delta^{k-i}u\right\vert ^{2}dx\\
&  =\int_{%
\mathbb{R}
^{2m}}%
{\displaystyle\sum\limits_{0\leq i,j\leq k}}
\left(  -1\right)  ^{k-i}\left(  -1\right)  ^{k-j}\binom{k}{i}\binom{k}{j}%
\tau^{i}\tau^{j}\Delta^{k-i}u\Delta^{k-j}udx\\
&  =%
{\displaystyle\sum\limits_{s=0}^{2k}}
{\displaystyle\sum\limits_{i+j=s}}
\left(  -1\right)  ^{k-i}\left(  -1\right)  ^{k-j}\binom{k}{i}\binom{k}{j}%
\tau^{i}\tau^{j}\int_{%
\mathbb{R}
^{2m}}\Delta^{k-i}u\Delta^{k-j}udx\\
&  =%
{\displaystyle\sum\limits_{s=0}^{2k}}
{\displaystyle\sum\limits_{i+j=s}}
\binom{k}{i}\binom{k}{j}\tau^{s}\int_{%
\mathbb{R}
^{2m}}\left\vert \nabla^{2k-s}u\right\vert ^{2}dx.
\end{align*}
From the coefficients of $x^{s}$ in the identity
\[
\left(  1+x\right)  ^{k}(1+x)^{k}=(1+x)^{2k}%
\]
we have
\[%
{\displaystyle\sum\limits_{i+j=s}}
\binom{k}{i}\binom{k}{j}=\binom{2k}{s}%
\]
and then
\begin{equation}
\int_{%
\mathbb{R}
^{2m}}\left\vert \left(  -\Delta+\tau I\right)  ^{k}u\right\vert ^{2}dx=\int_{%
\mathbb{R}
^{2m}}\left(
{\displaystyle\sum\limits_{j=0}^{m}}
\binom{m}{j}\tau^{m-j}\left\vert \nabla^{j}u\right\vert ^{2}\right)  dx.
\label{2.02}%
\end{equation}

From these observations, we have when $m=2k,$ $k\in%
\mathbb{N}
:$
\begin{equation}
\left\Vert \left(  -\Delta+\tau I\right)  ^{k}u\right\Vert _{2}=\left[
{\displaystyle\sum\limits_{j=0}^{m}}
\binom{m}{j}\tau^{m-j}\left\Vert \nabla^{j}u\right\Vert _{2}^{2}\right]
^{1/2}. \label{2.003}%
\end{equation}

From (\ref{2.01}), (\ref{2.02}) and (\ref{2.003}), we have

\begin{lemma}
Assume $m=2k,$ $k\in%
\mathbb{N}
$. Let $a_{0}=1,a_{1},...,a_{m}>0$. There exists a real number $\tau>0$ such
that for all $u\in W^{m,2}\left(
\mathbb{R}
^{2m}\right)  :$
\[
\left\Vert \left(  -\Delta+\tau I\right)  ^{k}u\right\Vert _{2}^{2}\leq%
{\displaystyle\sum\limits_{j=0}^{m}}
a_{m-j}\left\Vert \nabla^{j}u\right\Vert _{2}^{2}%
\]

\end{lemma}

\begin{proof}
We just need to choose $\tau>0$ such that
\[
\binom{m}{j}\tau^{m-j}\leq a_{m-j},~j=0,1,...,m.
\]

\end{proof}

When $m=2k+1,~k\in%
\mathbb{N}
$, we have%
\begin{equation}
\nabla\left(  -\Delta+\tau I\right)  ^{k}u=%
{\displaystyle\sum\limits_{i=0}^{k}}
\left(  -1\right)  ^{k-i}\binom{k}{i}\tau^{i}\nabla\Delta^{k-i}u \label{2.04}%
\end{equation}
where
\[
\binom{k}{j}=\frac{k!}{j!(k-j)!}.
\]
Similarly, we can prove that
\begin{equation}
\int_{%
\mathbb{R}
^{2m}}\left\vert \nabla\left(  -\Delta+\tau I\right)  ^{k}u\right\vert
^{2}dx=\int_{%
\mathbb{R}
^{2m}}\left(
{\displaystyle\sum\limits_{j=1}^{m}}
\binom{m-1}{j-1}\tau^{m-j}\left\vert \nabla^{j}u\right\vert ^{2}\right)  dx.
\label{2.05}%
\end{equation}

Thus, we have for $m=2k+1,$ $k\in%
\mathbb{N}
:$
\begin{equation}
\left\Vert \nabla\left(  -\Delta+\tau I\right)  ^{k}u\right\Vert _{2}=\left[
{\displaystyle\sum\limits_{j=1}^{m}}
\binom{m-1}{j-1}\tau^{m-j}\left\Vert \nabla^{j}u\right\Vert _{2}^{2}\right]
^{1/2}. \label{2.06}%
\end{equation}
By (\ref{2.003}) and (\ref{2.06}), we get
\begin{equation}
\left\Vert \nabla\left(  -\Delta+\tau I\right)  ^{k}u\right\Vert _{2}^{2}%
+\tau\left\Vert \left(  -\Delta+\tau I\right)  ^{k}u\right\Vert _{2}^{2}=%
{\displaystyle\sum\limits_{j=0}^{m}}
\binom{m}{j}\tau^{m-j}\left\Vert \nabla^{j}u\right\Vert _{2}^{2}. \label{2.07}%
\end{equation}

\begin{lemma}
Assume $m=2k+1,~k\in%
\mathbb{N}
$. Let $a_{0}=1,a_{1},...,a_{m}>0$. There exists a real number $\tau>0$ such
that for all $u\in W^{m,2}\left(
\mathbb{R}
^{2m}\right)  :$
\[
\left\Vert \nabla\left(  -\Delta+\tau I\right)  ^{k}u\right\Vert _{2}^{2}%
+\tau\left\Vert \left(  -\Delta+\tau I\right)  ^{k}u\right\Vert _{2}^{2}\leq%
{\displaystyle\sum\limits_{j=0}^{m}}
a_{m-j}\left\Vert \nabla^{j}u\right\Vert _{2}^{2}%
\]

\end{lemma}

\begin{proof}
Again, we just need to choose $\tau>0$ such that%
\[
\binom{m}{j}\tau^{m-j}\leq a_{m-j}.
\]

\end{proof}

In the general case, we have the following result

\begin{lemma}
Assume that $m$ is \textit{an odd integer less than }$n$: $m=2k+1$. There
exists a real number $C>0$ such that for all $u\in W^{m,\frac{n}{m}}\left(
\mathbb{R}
^{n}\right)  :$
\[
\left\Vert \nabla^{m}u\right\Vert _{\frac{n}{m}}^{\frac{n}{m}}+\frac{1}{C}%
{\displaystyle\sum\limits_{j=0}^{m-1}}
\left\Vert \nabla^{j}u\right\Vert _{\frac{n}{m}}^{\frac{n}{m}}\leq\left\Vert
\nabla\left(  -\Delta+I\right)  ^{k}u\right\Vert _{\frac{n}{m}}^{\frac{n}{m}%
}+\left\Vert \left(  -\Delta+I\right)  ^{k}u\right\Vert _{\frac{n}{m}}%
^{\frac{n}{m}}.
\]

\end{lemma}

We now introduce the Sobolev space of functions with homogeneous Navier
boundary conditions:%
\[
W_{N}^{m,2}\left(  B_{R}\right)  :=\left\{  u\in W^{m,2}:\Delta^{j}u=0\text{
on }\partial B_{R}\text{ for }0\leq j\leq\left[  \frac{m-1}{2}\right]
\right\}
\]
where $B_{R}=\left\{  x\in%
\mathbb{R}
^{2m}:\left\vert x\right\vert <R\right\}  .$ It is easy to see that
$W_{N}^{m,2}\left(  B_{R}\right)  $ contains $W_{0}^{m,2}\left(  B_{R}\right)
$ as a closed subspace. Also, we define
\begin{align*}
W_{rad}^{m,2}\left(  B_{R}\right)   &  :=\left\{  u\in W^{m,2}%
:u(x)=u(\left\vert x\right\vert )\text{ a.e. in }B_{R}\right\}  ,\\
W_{N,rad}^{m,2}\left(  B_{R}\right)   &  =W_{N}^{m,2}\left(  B_{R}\right)
\cap W_{rad}^{m,2}\left(  B_{R}\right)  .
\end{align*}

Finally, we give some radial lemmas which will be used in our proofs (see
\cite{BL, K, RS}):

\begin{lemma}
If $u\in W^{1,\frac{n}{m}}\left(
\mathbb{R}
^{n}\right)  $ then
\[
\left\vert u(x)\right\vert \leq\left(  \frac{1}{m\sigma_{n}}\right)
^{\frac{m}{n}}\frac{1}{\left\vert x\right\vert ^{\frac{n-1}{n}m}}\left\Vert
u\right\Vert _{W^{1,\frac{n}{m}}}%
\]
for a.e. $x\in%
\mathbb{R}
^{n}$, where $\sigma_{n}$ is the volume of the unit ball in $%
\mathbb{R}
^{n}$.
\end{lemma}

\begin{lemma}
If $u\in L^{p}\left(
\mathbb{R}
^{n}\right)  ,~1\leq p<\infty,$ is a radial non-increasing function, then
\[
\left\vert u(x)\right\vert \leq\left(  \frac{n}{\omega_{n-1}}\right)
^{\frac{1}{p}}\frac{1}{\left\vert x\right\vert ^{\frac{n}{p}}}\left\Vert
u\right\Vert _{L^{p}\left(
\mathbb{R}
^{n}\right)  }%
\]
for a.e. $x\in%
\mathbb{R}
^{n}$.
\end{lemma}

\section{An iterated comparison principle}

\bigskip In this section, we still denote by $B_{R}$ the set $\left\{  x\in%
\mathbb{R}
^{n}:~\left\vert x\right\vert <R\right\}  $ and $\left\vert B_{R}\right\vert $
the Lebesgue measure of $B_{R}$, namely $\left\vert B_{R}\right\vert
=\sigma_{n}R^{n}$ where $\sigma_{n}$ is the volume of the unit ball in $%
\mathbb{R}
^{n}$. Let $u:B_{R}\rightarrow%
\mathbb{R}
$ be a measurable function. The distribution function of $u$ is defined by
\[
\mu_{u}(t)=\left\vert \left\{  x\in B_{R}~|~\left\vert u(x)\right\vert
>t\right\}  \right\vert ,~\forall t\geq0.
\]
The decreasing rearrangement of $u$ is defined by
\[
u^{\ast}(s)=\inf\left\{  t\geq0:\mu_{u}(t)<s\right\}  ,~\forall s\in\left[
0,\left\vert B_{R}\right\vert \right]  ,
\]
and the spherically symmetric decreasing rearrangement of $u$ by%
\[
u^{\#}(x)=u^{\ast}\left(  \sigma_{n}\left\vert x\right\vert ^{n}\right)
~\forall x\in B_{R}.
\]
We have that $u^{\#}$ is the unique nonnegative integrable function which is
radially symmetric, nonincreasing and has the same distribution function as
$\left\vert u\right\vert .$

Let $\tau>0$ and $u$ be a weak solution of
\begin{equation}
\left\{
\begin{array}
[c]{c}%
-\Delta u+\tau u=f\text{ in }B_{R}\\
u\in W_{0}^{1,2}\left(  B_{R}\right)
\end{array}
\right.  \label{2.1}%
\end{equation}
where $f\in L^{\frac{2n}{n+2}}\left(  B_{R}\right)  $. We have the following
result that can be found in \cite{TV}:

\begin{proposition}
If $u$ is a nonnegative weak solution of (\ref{2.1}) then
\begin{equation}
-\frac{du^{\ast}}{ds}(s)\leq\frac{s^{\frac{2}{n}-2}}{n^{2}\sigma_{n}^{2/n}}%
{\displaystyle\int\limits_{0}^{s}}
\left(  f^{\ast}-\tau u^{\ast}\right)  d\tau,~\forall s\in\left(  0,\left\vert
B_{R}\right\vert \right)  . \label{2.2}%
\end{equation}

\end{proposition}

Now, we consider the problem
\begin{equation}
\left\{
\begin{array}
[c]{c}%
-\Delta v+\tau v=f^{\#}\text{ in }B_{R}\\
v\in W_{0}^{1,2}\left(  B_{R}\right)
\end{array}
\right.  \label{2.3}%
\end{equation}
Due to the radial symmetry of the equation, the unique solution $v$ of
(\ref{2.3}) is radially symmetric and we have
\[
-\frac{d\widehat{v}}{ds}(s)=\frac{s^{\frac{2}{n}-2}}{n^{2}\sigma_{n}^{2/n}}%
{\displaystyle\int\limits_{0}^{s}}
\left(  f^{\ast}-\tau\widehat{v}\right)  d\tau,~\forall s\in\left(
0,\left\vert B_{R}\right\vert \right)
\]
where $\widehat{v}\left(  \sigma_{n}\left\vert x\right\vert ^{n}\right)
:=v(x)$. We have the following comparison of integrals in balls that again can
be found in \cite{TV}:

\begin{proposition}
Let $u,~v$ be weak solutions of (\ref{2.1}) and (\ref{2.3}) respectively. For
every $r\in\left(  0,R\right)  $ we have%
\[
\int_{B_{r}}u^{\#}dx\leq\int_{B_{r}}vdx.
\]

\end{proposition}

We now apply the comparison principle for the polyharmonic operator. Let $u\in
W^{2k,2}\left(  B_{R}\right)  $ be a weak solution of
\begin{equation}
\left\{
\begin{array}
[c]{c}%
\left(  -\Delta+\tau I\right)  ^{k}u=f\text{ in }B_{R}\\
u\in W_{N}^{2k,2}\left(  B_{R}\right)
\end{array}
\right.  \label{2.4}%
\end{equation}
where $f\in L^{\frac{2n}{n+2}}\left(  B_{R}\right)  $. If we consider the
problem
\begin{equation}
\left\{
\begin{array}
[c]{c}%
\left(  -\Delta+\tau I\right)  ^{k}v=f^{\#}\text{ in }B_{R}\\
v\in W_{N}^{2k,2}\left(  B_{R}\right)
\end{array}
\right.  \label{2.5}%
\end{equation}
then we have the following comparison of integrals in balls:

\begin{proposition}
Let $u,~v$ be weak solutions of the polyharmonic problems (\ref{2.4}) and
(\ref{2.5}) respectively. For every $r\in\left(  0,R\right)  $ we have%
\[
\int_{B_{r}}u^{\#}dx\leq\int_{B_{r}}vdx.
\]

\end{proposition}

\begin{proof}
Since equations in (\ref{2.4}) and (\ref{2.5}) are considered with homogeneous
Navier boundary conditions, they may be rewritten as second order systems:%

\begin{align*}
(P1)\left\{
\begin{array}
[c]{c}%
-\Delta u_{1}+\tau u_{1}=f\text{ in }B_{R}\\
u_{1}\in W_{0}^{1,2}\left(  B_{R}\right)
\end{array}
\right.  ~\ (Pi)\left\{
\begin{array}
[c]{c}%
-\Delta u_{i}+\tau u_{i}=u_{i-1}\text{ in }B_{R}\\
u_{i}\in W_{0}^{1,2}\left(  B_{R}\right)
\end{array}
\right.  ~i  &  \in\left\{  2,3,...,k\right\} \\
(Q1)\left\{
\begin{array}
[c]{c}%
-\Delta v_{1}+\tau v_{1}=f^{\#}\text{ in }B_{R}\\
v_{1}\in W_{0}^{1,2}\left(  B_{R}\right)
\end{array}
\right.  ~\ (Qi)\left\{
\begin{array}
[c]{c}%
-\Delta v_{i}+\tau v_{i}=v_{i-1}\text{ in }B_{R}\\
v_{i}\in W_{0}^{1,2}\left(  B_{R}\right)
\end{array}
\right.  ~i  &  \in\left\{  2,3,...,k\right\}
\end{align*}
where $u_{k}=u$ and $v_{k}=v$. Thus we have to prove that for every
$r\in\left(  0,R\right)  $%
\begin{equation}
\int_{B_{r}}u_{k}^{\#}dx\leq\int_{B_{r}}v_{k}dx. \label{2.6}%
\end{equation}
By the above proposition (Proposition 3.2), we have
\[
\int_{B_{r}}u_{1}^{\#}dx\leq\int_{B_{r}}v_{1}dx.
\]
Now, if we have
\[
\int_{B_{r}}u_{j}^{\#}dx\leq\int_{B_{r}}v_{j}dx\text{ for all }j=1,...,i,
\]
we will prove that
\[
\int_{B_{r}}u_{i+1}^{\#}dx\leq\int_{B_{r}}v_{i+1}dx.
\]
Without loss of generality, we may assume that $u_{i+1}\geq0$. In fact, let
$\overline{u}_{i+1}$ be a weak solution of
\[
\left\{
\begin{array}
[c]{c}%
-\Delta\overline{u}_{i+1}+\tau\overline{u}_{i+1}=\left\vert u_{i}\right\vert
\text{ in }B_{R}\\
\overline{u}_{i+1}\in W_{0}^{1,2}\left(  B_{R}\right)
\end{array}
\right.
\]
then the maximum principle implies that $\overline{u}_{i+1}\geq0$ and
$\overline{u}_{i+1}\geq\left\vert u_{i+1}\right\vert $ in $B_{R}$.

Since $u_{i+1}$ is a nonnegative weak solution of $(P\left(  i+1\right)  )$
and $v_{i+1}$ is a nonnegative weak solution of $(Q\left(  i+1\right)  )$ then
by Proposition 3.1 we have%
\begin{align*}
-\frac{du_{i+1}^{\ast}}{ds}(s)  &  \leq\frac{s^{\frac{2}{n}-2}}{n^{2}%
\sigma_{n}^{2/n}}%
{\displaystyle\int\limits_{0}^{s}}
\left(  u_{i}^{\ast}-\tau u_{i+1}^{\ast}\right)  d\tau,~\forall s\in\left(
0,\left\vert B_{R}\right\vert \right)  ,\\
-\frac{d\widehat{v}_{i+1}}{ds}(s)  &  =\frac{s^{\frac{2}{n}-2}}{n^{2}%
\sigma_{n}^{2/n}}%
{\displaystyle\int\limits_{0}^{s}}
\left(  \widehat{v}_{i}-\tau\widehat{v}_{i+1}\right)  d\tau,~\forall
s\in\left(  0,\left\vert B_{R}\right\vert \right)
\end{align*}
Thus for all $s\in\left(  0,\left\vert B_{R}\right\vert \right)  $%
\[
\frac{d\widehat{v}_{i+1}}{ds}(s)-\frac{du_{i+1}^{\ast}}{ds}(s)-\frac
{s^{\frac{2}{n}-2}}{n^{2}\sigma_{n}^{2/n}}%
{\displaystyle\int\limits_{0}^{s}}
\left(  \tau\widehat{v}_{i+1}-\tau u_{i+1}^{\ast}\right)  d\tau\leq
\frac{s^{\frac{2}{n}-2}}{n^{2}\sigma_{n}^{2/n}}%
{\displaystyle\int\limits_{0}^{s}}
\left(  u_{i}^{\ast}-\widehat{v}_{i}\right)  d\tau.
\]
Using the induction hypotheses, we get that
\[%
{\displaystyle\int\limits_{0}^{s}}
\left(  u_{i}^{\ast}-\widehat{v}_{i}\right)  d\tau\leq0~\forall s\in\left(
0,\left\vert B_{R}\right\vert \right)
\]
and then
\[
\frac{d\widehat{v}_{i+1}}{ds}(s)-\frac{du_{i+1}^{\ast}}{ds}(s)-\frac
{s^{\frac{2}{n}-2}}{n^{2}\sigma_{n}^{2/n}}%
{\displaystyle\int\limits_{0}^{s}}
\left(  \tau\widehat{v}_{i+1}-\tau u_{i+1}^{\ast}\right)  d\tau\leq0.
\]
Setting
\[
y(s)=%
{\displaystyle\int\limits_{0}^{s}}
\left(  \widehat{v}_{i+1}-u_{i+1}^{\ast}\right)  d\tau~~\,\,\forall
s\in\left(  0,\left\vert B_{R}\right\vert \right)
\]
we get
\[
\left\{
\begin{array}
[c]{c}%
y^{\prime\prime}-\frac{\tau s^{\frac{2}{n}-2}}{n^{2}\sigma_{n}^{2/n}}%
y\leq0,~\forall s\in\left(  0,\left\vert B_{R}\right\vert \right) \\
y(0)=y^{\prime}(\left\vert B_{R}\right\vert )=0
\end{array}
\right.  .
\]
By maximum principle, we have that $y\geq0$ which is the desired result.
\end{proof}

From the above proposition, we have the following corollary:

\begin{corollary}
Let $u,~v$ be weak solutions of the polyharmonic problems (\ref{2.4}) and
(\ref{2.5}) respectively. For every convex nondecreasing function
$\phi:\left[  0,+\infty\right)  \rightarrow\left[  0,+\infty\right)  $ we have%
\[
\int_{B_{r}}\phi\left(  \left\vert u\right\vert \right)  dx\leq\int_{B_{r}%
}\phi\left(  \left\vert v\right\vert \right)  dx.
\]

\end{corollary}

\begin{remark}
If $f\in C_{0}^{\infty}\left(
\mathbb{R}
^{n}\right)  $, $suppf\subset B_{R}$, then we can conclude that $u$ and $v$ in
Proposition 3.3 belong to $W_{N}^{m,\frac{n}{m}}\left(  B_{R}\right)  $ with
$m=2k$ or $2k+1.$
\end{remark}

\section{Proofs of Theorems 1.2, 1.3 and 1.4 when $m$ is even}

In this section, we will prove Theorem 1.2 in the case when $m$ is even,
namely, $m=2k,~k\in%
\mathbb{N}
$.

\begin{theorem}
Let $m=2k,~k\in%
\mathbb{N}
$. For all $\tau>0$, there holds%
\[
\underset{u\in W^{m,2}\left(
\mathbb{R}
^{2m}\right)  ,\left\Vert \left(  -\Delta+\tau I\right)  ^{k}u\right\Vert
_{2}\leq1}{\sup}\int_{%
\mathbb{R}
^{2m}}\left(  e^{\beta(2m,m)u^{2}}-1\right)  dx<\infty.
\]
Furthermore this inequality is sharp, i.e., if $\beta(2m,m)$ is
replaced by any $\beta>\beta(2m,m)$, then the supremum is infinite.
\end{theorem}

\begin{proof}
Let $u\in W^{m,2}\left(
\mathbb{R}
^{2m}\right)  ,\left\Vert \left(  -\Delta+\tau I\right)  ^{k}u\right\Vert
_{2}\leq1$, by the fact that $C_{0}^{\infty}\left(
\mathbb{R}
^{2m}\right)  $ is dense in $W^{m,2}\left(
\mathbb{R}
^{2m}\right)  $, without loss of generality, we can find a sequence of
functions $u_{l}\in C_{0}^{\infty}\left(
\mathbb{R}
^{2m}\right)  $ such that $u_{l}\rightarrow u$ in $W^{m,2}\left(
\mathbb{R}
^{2m}\right)  $ and $\int_{%
\mathbb{R}
^{2m}}\left\vert \left(  -\Delta+\tau I\right)  ^{k}u_{l}\right\vert
^{2}dx\leq1$ and suppose that supp\thinspace\ $u_{l}\subset B_{R_{l}}$ for any
fixed $l$. Let $f_{l}:=\left(  -\Delta+\tau I\right)  ^{k}u_{l}$. Consider the
problem%
\[
\left\{
\begin{array}
[c]{c}%
\left(  -\Delta+\tau I\right)  ^{k}v_{l}=f_{l}^{\#}\\
v_{l}\in W_{N}^{m,2}\left(  B_{R_{l}}\right)
\end{array}
\right.  .
\]
By the property of rearrangement, we have
\begin{equation}
\int_{B_{R_{l}}}\left\vert \left(  -\Delta+\tau I\right)  ^{k}v_{l}\right\vert
^{2}dx=\int_{B_{R_{l}}}\left\vert \left(  -\Delta+\tau I\right)  ^{k}%
u_{l}\right\vert ^{2}dx\leq1\label{3.1}%
\end{equation}
and by Corollary 3.1, we get
\[
\int_{B_{R_{l}}}\left(  e^{\beta_{0}u_{l}^{2}}-1\right)  dx=\int_{B_{R_{l}}%
}\left(  e^{\beta_{0}u_{l}^{\#2}}-1\right)  dx\leq\int_{B_{R_{l}}}\left(
e^{\beta_{0}v_{l}^{2}}-1\right)  dx.
\]

Also, from (\ref{3.1}) and (\ref{2.003}), we have%
\begin{align*}
\left\Vert v_{l}\right\Vert _{W^{1,2}} &  =\left[  \int_{B_{R_{l}}}\left(
\left\vert v_{l}\right\vert ^{2}+\left\vert \nabla v_{l}\right\vert
^{2}\right)  \right]  ^{1/2}\\
&  \leq\sqrt{\frac{1}{\tau^{m}}+\frac{1}{m\tau^{m-1}}}.
\end{align*}
Now, writing
\begin{align*}
\int_{B_{R_{l}}}\left(  e^{\beta_{0}v_{l}^{2}}-1\right)  dx &  \leq
\int_{B_{R_{0}}}\left(  e^{\beta_{0}v_{l}^{2}}-1\right)  dx+\int_{B_{R_{l}%
}\smallsetminus B_{R_{0}}}\left(  e^{\beta_{0}v_{l}^{2}}-1\right)  dx\\
&  =I_{1}+I_{2}%
\end{align*}
where $R_{0}$ depends only on $\tau$ and will be chosen later, we will prove
that both $I_{1}$ and $I_{2}$ are bounded uniformly by a constant that depends
only on $\tau$.

Using Theorem B, we can estimate $I_{1}$. Indeed, we just need to construct an
auxiliary radial function $w_{l}\in W_{N}^{m,2}\left(  B_{R_{0}}\right)  $
with $\left\Vert \nabla^{m}w_{l}\right\Vert _{2}\leq1$ which increases the
integral we are interested in. Such a function was constructed in \cite{RS}.
For the completeness, we give the detail here. For each $i\in\left\{
1,2,...k-1\right\}  $ we define
\[
g_{i}\left(  \left\vert x\right\vert \right)  :=\left\vert x\right\vert
^{m-2i},~\forall x\in B_{R_{0}}%
\]
so $g_{i}\in W_{rad}^{m,2}\left(  B_{R_{0}}\right)  $. Moreover,%
\[
\Delta^{j}g_{i}\left(  \left\vert x\right\vert \right)  =\left\{
\begin{array}
[c]{c}%
c_{i}^{j}\left\vert x\right\vert ^{m-2(i+j)}\text{ for }j\in\left\{
1,2,...k-i\right\}  \\
0\text{ for }j\in\left\{  k-i+1,...,k\right\}
\end{array}
\right.  ~\forall x\in B_{R_{0}}%
\]
where
\[
c_{i}^{j}=%
{\displaystyle\prod\limits_{h=1}^{j}}
\left[  n+m-2\left(  h+i\right)  \right]  \left[  m-2\left(  i+h-1\right)
\right]  \text{, }\forall j\in\left\{  1,2,...k-i\right\}  .
\]
Let
\[
z_{l}\left(  \left\vert x\right\vert \right)  :=v_{l}\left(  \left\vert
x\right\vert \right)  -%
{\displaystyle\sum\limits_{i=1}^{k-1}}
a_{i}g_{i}\left(  \left\vert x\right\vert \right)  -a_{k}~\forall x\in
B_{R_{0}}%
\]
where
\begin{align*}
a_{i} &  :=\frac{\Delta^{k-i}v_{l}\left(  R_{0}\right)  -%
{\displaystyle\sum\limits_{j=1}^{i-1}}
a_{j}\Delta^{k-i}g_{j}\left(  R_{0}\right)  }{\Delta^{k-i}g_{i}\left(
R_{0}\right)  },~\forall i\in\left\{  1,2,...k-1\right\}  ,\\
a_{k} &  :=v_{l}\left(  R_{0}\right)  -%
{\displaystyle\sum\limits_{i=1}^{k-1}}
a_{i}g_{i}\left(  R_{0}\right)  .
\end{align*}
We can check that (see \cite{RS})
\begin{align*}
z_{l} &  \in W_{N,rad}^{m,2}\left(  B_{R_{0}}\right)  ,\\
\nabla^{m}v_{l} &  =\nabla^{m}z_{l}\text{ in }B_{R_{0}}\text{.}%
\end{align*}
We have the following lemma whose proof can be found in \cite{RS}:

\begin{lemma}
For $0<\left\vert x\right\vert \leq R_{0}$ we have for some $d(m, R_{0})$ only
depending on $m$ and $R_{0}$ such that
\[
\left\vert v_{l}\left(  \left\vert x\right\vert \right)  \right\vert ^{2}%
\leq\left\vert z_{l}\left(  \left\vert x\right\vert \right)  \right\vert
^{2}\left(  1+c_{m}%
{\displaystyle\sum\limits_{j=1}^{k-1}}
\frac{1}{R_{0}^{4j-1}}\left\Vert \Delta^{k-j}v_{l}\right\Vert _{W^{1,2}}%
^{2}+\frac{c_{m}}{R_{0}^{2m-1}}\left\Vert v_{l}\right\Vert _{W^{1,2}}%
^{2}\right)  ^{2}+d(m,R_{0}).
\]

\end{lemma}

Now, setting
\[
w_{l}\left(  \left\vert x\right\vert \right)  :=z_{l}\left(  \left\vert
x\right\vert \right)  \left(  1+c_{m}%
{\displaystyle\sum\limits_{j=1}^{k-1}}
\frac{1}{R_{0}^{4j-1}}\left\Vert \Delta^{k-j}v_{l}\right\Vert _{W^{1,2}}%
^{2}+\frac{c_{m}}{R_{0}^{2m-1}}\left\Vert v_{l}\right\Vert _{W^{1,2}}%
^{2}\right)  .
\]
Since
\begin{align*}
z_{l}  &  \in W_{N,rad}^{m,2}\left(  B_{R_{0}}\right)  ,\\
\nabla^{m}v_{l}  &  =\nabla^{m}z_{l}\text{ in }B_{R_{0}}\text{.}%
\end{align*}
we have
\[
w_{l}\in W_{N,rad}^{m,2}\left(  B_{R_{0}}\right)
\]
and%
\[
\left\Vert \nabla^{m}w_{l}\right\Vert _{2}=\left\Vert \nabla^{m}%
z_{l}\right\Vert _{2}\left(  1+c_{m}%
{\displaystyle\sum\limits_{j=1}^{k-1}}
\frac{1}{R_{0}^{4j-1}}\left\Vert \Delta^{k-j}v_{l}\right\Vert _{W^{1,2}}%
^{2}+\frac{c_{m}}{R_{0}^{2m-1}}\left\Vert v_{l}\right\Vert _{W^{1,2}}%
^{2}\right)  .
\]
Note that
\begin{align*}
\left\Vert \nabla^{m}z_{l}\right\Vert _{2}  &  =\left\Vert \nabla^{m}%
v_{l}\right\Vert _{2}\\
&  \leq\left(  1-\lambda%
{\displaystyle\sum\limits_{j=1}^{k-1}}
\left\Vert \Delta^{k-j}v_{l}\right\Vert _{W^{1,2}}^{2}-\lambda\left\Vert
v_{l}\right\Vert _{W^{1,2}}^{2}\right)  ^{1/2}\\
&  \leq1-\frac{\lambda}{2}%
{\displaystyle\sum\limits_{j=1}^{k-1}}
\left\Vert \Delta^{k-j}v_{l}\right\Vert _{W^{1,2}}^{2}-\frac{\lambda}%
{2}\left\Vert v_{l}\right\Vert _{W^{1,2}}^{2}%
\end{align*}
where
\[
\lambda=\min\left\{  \binom{m}{j}\tau^{m-j}:j=0,1,...,m-1.\right\}
\]
we have%
\begin{align*}
\left\Vert \nabla^{m}w_{l}\right\Vert _{2}  &  \leq\left(  1-\frac{\lambda}{2}%
{\displaystyle\sum\limits_{j=1}^{k-1}}
\left\Vert \Delta^{k-j}v_{l}\right\Vert _{W^{1,2}}^{2}-\frac{\lambda}%
{2}\left\Vert v_{l}\right\Vert _{W^{1,2}}^{2}\right)  \times\\
&  \times\left(  1+c_{m}%
{\displaystyle\sum\limits_{j=1}^{k-1}}
\frac{1}{R_{0}^{4j-1}}\left\Vert \Delta^{k-j}v_{l}\right\Vert _{W^{1,2}}%
^{2}+\frac{c_{m}}{R_{0}^{2m-1}}\left\Vert v_{l}\right\Vert _{W^{1,2}}%
^{2}\right) \\
&  \leq1+%
{\displaystyle\sum\limits_{j=1}^{k-1}}
\left(  \frac{c_{m}}{R_{0}^{4j-1}}-\frac{\lambda}{2}\right)  \left\Vert
\Delta^{k-j}v_{l}\right\Vert _{W^{1,2}}^{2}+\left(  \frac{c_{m}}{R_{0}^{2m-1}%
}-\frac{\lambda}{2}\right)  \left\Vert v_{l}\right\Vert _{W^{1,2}}^{2}\\
&  \leq1
\end{align*}
if we choose $R_{0}=R_{0}(\tau)$ sufficiently large.

Finally, note that
\[
I_{1}\leq e^{\beta_{0}d(m,R_{0})}\int_{B_{R_{0}}}e^{\beta_{0}w_{l}^{2}}dx,
\]
using Theorem B, we can conclude that $I_{1}$ is bounded by a constant
depending only on $\tau$ since $\left\Vert \nabla^{m}w_{l}\right\Vert _{2}%
\leq1$ and $w_{l}\in W_{N,rad}^{m,2}\left(  B_{R_{0}}\right)  .$

Now, we will estimate $I_{2}$. We choose $R_{0}\geq\left[  \frac{1}%
{m\sigma_{n}}\left(  \frac{1}{\tau^{m}}+\frac{1}{m\tau^{m-1}}\right)  \right]
$ $^{\frac{1}{n-1}}$ then from the Radial Lemma 2.4 we get that $\left\vert
v_{l}(x)\right\vert \leq1$ when $\left\vert x\right\vert \geq R_{0}$. Thus we
have%
\begin{align*}
I_{2}  &  =\int_{B_{R_{l}}\smallsetminus B_{R_{0}}}\left(  e^{\beta_{0}%
v_{l}^{2}}-1\right)  dx\\
&  \leq%
{\displaystyle\sum\limits_{j=1}^{\infty}}
\frac{\beta_{0}^{j}}{j!}\int_{B_{R_{l}}}v_{l}^{2}\\
&  \leq\frac{1}{\tau^{m}}%
{\displaystyle\sum\limits_{j=1}^{\infty}}
\frac{\beta_{0}^{j}}{j!}.
\end{align*}
Thus we have that $\int_{B_{R_{l}}}\left(  e^{\beta_{0}v_{l}^{2}}-1\right)
dx$ is bounded by a constant depending only on $\tau$.

Combining the above estimates and using Fatou's lemma, we can conclude that
\[
\underset{\left\Vert \left(  -\Delta+\tau I\right)  ^{k}u\right\Vert _{2}%
\leq1}{\sup}\int_{%
\mathbb{R}
^{2m}}\left(  e^{\beta_{0}u^{2}}-1\right)  dx<\infty\text{.}%
\]

When $\beta>\beta_{0}$, it's easy to check that the sequence given by Ruf and
Sani (see Proposition 6.2. in \cite{RS}) will make our supremum blow up and we
then complete the proof of Theorem 4.1.
\end{proof}

\bigskip\textbf{Proof of Theorem 1.3 when $m$ is even: }Choose $\tau>0$ as in
Lemma 2.1, we have
\[
\int_{%
\mathbb{R}
^{2m}}\left(
{\displaystyle\sum\limits_{j=0}^{m}}
a_{m-j}\left\vert \nabla^{j}u\right\vert ^{2}\right)  dx\geq\left\Vert \left(
-\Delta+\tau I\right)  ^{k}u\right\Vert _{2}^{2}%
\]
and then%
\[
\underset{\int_{%
\mathbb{R}
^{2m}}\left(
{\displaystyle\sum\limits_{j=0}^{m}}
a_{m-j}\left\vert \nabla^{j}u\right\vert ^{2}\right)  dx\leq1}{\sup}\int_{%
\mathbb{R}
^{2m}}\left(  e^{\beta_{0}u^{2}}-1\right)  dx\leq\underset{\left\Vert \left(
-\Delta+\tau I\right)  ^{k}u\right\Vert _{2}\leq1}{\sup}\int_{%
\mathbb{R}
^{2m}}\left(  e^{\beta_{0}u^{2}}-1\right)  dx.
\]

Furthermore, we can check that the sequence given by Ruf and Sani (see
Proposition 6.2. in \cite{RS}) will make the supremum in Theorem 1.3 becomes
infinite and we complete the proof of Theorem 1.3.

\textbf{Proof of Theorem 1.4 when $m$ is even:} If we choose $a_{i}%
=1,\ i=0,1,...,m$, then by Lemma 2.1 we have proved Theorem 1.4 in the case
$m=2k$.

\section{\bigskip Proofs of Theorems 1.2, 1.3 and 1.4 when $m$ is odd}

\subsection{Proofs of Theorems 1.2, 1.3 and 1.4 when $n=2m=6.$}

For the convenience, first, we will prove Theorem 1.2 in the special case
$k=1$, i.e., we will prove that all $\tau>0$, there holds
\[
\underset{u\in W^{3,2}\left(
\mathbb{R}
^{6}\right)  ,~\left\Vert \nabla\left(  -\Delta+\tau I\right)  u\right\Vert
_{2}^{2}+\tau\left\Vert \left(  -\Delta+\tau I\right)  u\right\Vert _{2}%
^{2}\leq1}{\sup}\int_{%
\mathbb{R}
^{6}}\left(  e^{\beta_{0}u^{2}}-1\right)  dx<\infty\text{.}%
\]
where $\beta_{0}=\beta\left(  6,3\right)  .$

\begin{proof}
Let $u\in W^{3,2}\left(
\mathbb{R}
^{6}\right)  $ be such that
\[
\left\Vert \nabla\left(  -\Delta+\tau I\right)  u\right\Vert _{2}^{2}%
+\tau\left\Vert \left(  -\Delta+\tau I\right)  u\right\Vert _{2}^{2}\leq1.
\]
Again, by the density of $C_{0}^{\infty}\left(
\mathbb{R}
^{6}\right)  $ in $W^{3,2}\left(
\mathbb{R}
^{6}\right)  $, there exists a sequence of functions $u_{l}\in C_{0}^{\infty
}\left(
\mathbb{R}
^{6}\right)  $: $u_{l}\rightarrow u$ in $W^{3,2}\left(
\mathbb{R}
^{6}\right)  $,
\[
\int_{%
\mathbb{R}
^{6}}\left\vert \nabla\left(  -\Delta+\tau I\right)  u_{l}\right\vert
^{2}+\tau\left\vert \left(  -\Delta+\tau I\right)  u_{l}\right\vert ^{2}%
dx\leq1
\]
and supp\thinspace$u_{l}\subset B_{R_{l}}$ for any fixed $l$. Set
$f_{l}:=\left(  -\Delta+\tau I\right)  u_{l}$ and consider the problem%
\[
\left\{
\begin{array}
[c]{c}%
\left(  -\Delta+\tau I\right)  v_{l}=f_{l}^{\#}\\
v_{l}\in W_{N}^{3,2}\left(  B_{R_{l}}\right)
\end{array}
\right.  .
\]
By the properties of rearrangement, we have
\begin{align*}
\int_{B_{R_{l}}}\left\vert f_{l}\right\vert ^{2}dx &  =\int_{B_{R_{l}}%
}\left\vert f_{l}^{\#}\right\vert ^{2}dx\\
\int_{B_{R_{l}}}\left\vert \nabla f_{l}^{\#}\right\vert ^{2}dx &  \leq
\int_{B_{R_{l}}}\left\vert \nabla f_{l}\right\vert ^{2}dx
\end{align*}
which thus
\begin{align*}
\int_{B_{R_{l}}}\left\vert \left(  -\Delta+\tau I\right)  v_{l}\right\vert
^{2}dx &  =\int_{B_{R_{l}}}\left\vert \left(  -\Delta+\tau I\right)
u_{l}\right\vert ^{2}dx\\
\int_{B_{R_{l}}}\left\vert \nabla\left(  -\Delta+\tau I\right)  v_{l}%
\right\vert ^{2}dx &  \leq\int_{B_{R_{l}}}\left\vert \left(  -\Delta+\tau
I\right)  u_{l}\right\vert ^{2}dx
\end{align*}
So, we have
\begin{align}
&  \int_{\mathbb{R}^{6}}\left\vert \nabla\left(  -\Delta+\tau I\right)
v_{l}\right\vert ^{2}+\tau\left\vert \left(  -\Delta+\tau I\right)
v_{l}\right\vert ^{2}dx\nonumber\\
\leq &  \int_{%
\mathbb{R}
^{6}}\left\vert \nabla\left(  -\Delta+\tau I\right)  u_{l}\right\vert
^{2}+\tau\left\vert \left(  -\Delta+\tau I\right)  u_{l}\right\vert
^{2}dx\label{7.1}\\
\leq &  1.\nonumber
\end{align}
By the comparison argument (Corollary 3.1), we have%

\[
\int_{B_{R_{l}}}\left(  e^{\beta_{0}u_{l}^{2}}-1\right)  dx= \int_{B_{R_{l}}%
}\left(  e^{\beta_{0}u_{l}^{\#2}}-1\right)  dx\leq\int_{B_{R_{l}}}\left(
e^{\beta_{0}v_{l}^{2}}-1\right)  dx
\]
Recall
\begin{align}
&  \int_{%
\mathbb{R}
^{6}}\left\vert \nabla\left(  -\Delta+\tau I\right)  v_{l}\right\vert
^{2}+\tau\left\vert \left(  -\Delta+\tau I\right)  v_{l}\right\vert
^{2}dx\nonumber\\
=  &  \left\Vert \nabla^{3}v_{l}\right\Vert _{2}^{2}+3\tau\left\Vert \Delta
v_{l}\right\Vert _{2}^{2}+3\tau^{2}\left\Vert \nabla v_{l}\right\Vert _{2}%
^{2}+\tau^{3}\left\Vert v_{l}\right\Vert _{2}^{2}. \label{7.2}%
\end{align}

From (\ref{7.1}) and (\ref{7.2}), we have%
\begin{align*}
\left\Vert v_{l}\right\Vert _{W^{1,2}}  &  =\left[  \int_{B_{R_{l}}}\left(
\left\vert v_{l}\right\vert ^{2}+\left\vert \nabla v_{l}\right\vert
^{2}\right)  \right]  ^{1/2}\\
&  \leq\sqrt{\frac{1}{\tau^{3}}+\frac{1}{3\tau^{2}}}.
\end{align*}

Now, write
\begin{align*}
\int_{B_{R_{l}}}\left(  e^{\beta_{0}v_{l}^{2}}-1\right)  dx  &  \leq
\int_{B_{R_{0}}}\left(  e^{\beta_{0}v_{l}^{2}}-1\right)  dx+\int_{B_{R_{l}%
}\smallsetminus B_{R_{0}}}\left(  e^{\beta_{0}v_{l}^{2}}-1\right)  dx\\
&  =I_{1}+I_{2}%
\end{align*}
where $R_{0}$ depends only on $\tau$ and will be chosen later. We will prove
that both $I_{1}$ and $I_{2}$ are bounded uniformly by a constant that depends
only on $\tau$.

First, we will prove that $I_{1}$ is bounded by a constant depending only on
$\tau$ using Theorem B. In order to do that, we will construct an auxiliary
radial function $w_{l}$ such that $w_{l}\in$ $W_{N}^{3,2}\left(  B_{R_{0}%
}\right)  $, $\left\Vert \nabla^{3}w_{l}\right\Vert _{2}\leq1$ and
\[
\int_{B_{R_{0}}}\left(  e^{\beta_{0}v_{l}^{2}}-1\right)  dx\leq C(R_{0}%
)\int_{B_{R_{0}}}e^{\beta_{0}w_{l}^{2}}dx.
\]

The way to construct this radial function $w_{l}$ is very similar to the case
when $m$ is even. Let
\[
z_{l}\left(  \left\vert x\right\vert \right)  =v_{l}\left(  \left\vert
x\right\vert \right)  -\frac{\Delta v_{l}\left(  R_{0}\right)  }{12}\left\vert
x\right\vert ^{2}+\frac{R_{0}^{2}\Delta v_{l}\left(  R_{0}\right)  }{12}%
-v_{l}\left(  R_{0}\right)  ,~\forall x\in B_{R_{0}}%
\]
then $z_{l}\in W_{N,rad}^{3,2}\left(  B_{R_{0}}\right)  $. Similar to that in
the proof of Lemma 4.1, and by a combination of Radial Lemmas 2.4 and 2.5, we
can prove that for $0<\left\vert x\right\vert \leq R_{0}~$($R_{0}>1$), there
exists a universal constant $c>0$ and a positive constant $d(R_{0})$ depending
only on $R_{0}$ such that
\begin{equation}
\left\vert v_{l}\left(  \left\vert x\right\vert \right)  \right\vert ^{2}%
\leq\left\vert z_{l}\left(  \left\vert x\right\vert \right)  \right\vert
^{2}\left(  1+c\frac{1}{R_{0}}\left\Vert \Delta v_{l}\right\Vert _{2}%
^{2}+\frac{c}{R_{0}}\left\Vert v_{l}\right\Vert _{W^{1,2}}^{2}\right)
^{2}+d(R_{0}). \label{7.3}%
\end{equation}
Indeed, we have
\begin{align*}
v_{l}\left(  \left\vert x\right\vert \right)   &  =z_{l}\left(  \left\vert
x\right\vert \right)  +\frac{\Delta v_{l}\left(  R_{0}\right)  }{12}\left\vert
x\right\vert ^{2}-\frac{R_{0}^{2}\Delta v_{l}\left(  R_{0}\right)  }{12}%
+v_{l}\left(  R_{0}\right) \\
&  =z_{l}\left(  \left\vert x\right\vert \right)  +g\left(  \left\vert
x\right\vert \right)  ,~\forall x\in B_{R_{0}}%
\end{align*}
where
\[
g\left(  \left\vert x\right\vert \right)  =\frac{\Delta v_{l}\left(
R_{0}\right)  }{12}\left\vert x\right\vert ^{2}-\frac{R_{0}^{2}\Delta
v_{l}\left(  R_{0}\right)  }{12}+v_{l}\left(  R_{0}\right)  .
\]
Then
\begin{align*}
\left\vert v_{l}\left(  \left\vert x\right\vert \right)  \right\vert ^{2}  &
=\left[  z_{l}\left(  \left\vert x\right\vert \right)  +g\left(  \left\vert
x\right\vert \right)  \right]  ^{2}\\
&  =\left\vert z_{l}\left(  \left\vert x\right\vert \right)  \right\vert
^{2}+2z_{l}\left(  \left\vert x\right\vert \right)  g\left(  \left\vert
x\right\vert \right)  +\left\vert g\left(  \left\vert x\right\vert \right)
\right\vert ^{2}\\
&  \leq\left\vert z_{l}\left(  \left\vert x\right\vert \right)  \right\vert
^{2}+\left\vert z_{l}\left(  \left\vert x\right\vert \right)  \right\vert
^{2}\left\vert g\left(  \left\vert x\right\vert \right)  \right\vert
^{2}+1+\left\vert g\left(  \left\vert x\right\vert \right)  \right\vert ^{2}\\
&  =\left\vert z_{l}\left(  \left\vert x\right\vert \right)  \right\vert
^{2}\left(  1+\left\vert g\left(  \left\vert x\right\vert \right)  \right\vert
^{2}\right)  +1+\left\vert g\left(  \left\vert x\right\vert \right)
\right\vert ^{2}.
\end{align*}
Note that for $0<\left\vert x\right\vert \leq R_{0}~$($R_{0}>1$), we have by
Radial lemmas 2.4 and 2.5:
\begin{align*}
\left\vert g\left(  \left\vert x\right\vert \right)  \right\vert  &
=\frac{\Delta v_{l}\left(  R_{0}\right)  }{12}\left\vert x\right\vert
^{2}-\frac{R_{0}^{2}\Delta v_{l}\left(  R_{0}\right)  }{12}+v_{l}\left(
R_{0}\right) \\
&  \leq\left\vert \frac{\Delta v_{l}\left(  R_{0}\right)  }{12}\right\vert
\left\vert x\right\vert ^{2}+\left\vert \frac{R_{0}^{2}\Delta v_{l}\left(
R_{0}\right)  }{12}\right\vert +\left\vert v_{l}\left(  R_{0}\right)
\right\vert \\
&  \leq cR_{0}^{2}\frac{1}{R_{0}^{3}}\left\Vert \Delta v_{l}\right\Vert
_{2}+c\frac{1}{R_{0}^{5/2}}\left\Vert v_{l}\right\Vert _{W^{1,2}}^{2}.
\end{align*}
Thus (\ref{7.3}) follows.

Setting
\[
w_{l}\left(  \left\vert x\right\vert \right)  :=z_{l}\left(  \left\vert
x\right\vert \right)  \left(  1+c\frac{1}{R_{0}}\left\Vert \Delta
v_{l}\right\Vert _{2}^{2}+\frac{c}{R_{0}}\left\Vert v_{l}\right\Vert
_{W^{1,2}}^{2}\right)  ,~\forall x\in B_{R_{0}}%
\]
then it is clear that $w_{l}\in W_{N,rad}^{3,2}\left(  B_{R_{0}}\right) .$
Moreover, we have the following inequalities
\begin{align*}
\left\Vert \nabla^{3}w_{l}\right\Vert _{2}  &  =\left\Vert \nabla^{3}%
z_{l}\right\Vert _{2}\left(  1+c\frac{1}{R_{0}}\left\Vert \Delta
v_{l}\right\Vert _{2}^{2}+\frac{c}{R_{0}}\left\Vert v_{l}\right\Vert
_{W^{1,2}}^{2}\right) \\
&  =\left\Vert \nabla^{3}v_{l}\right\Vert _{2}\left(  1+c\frac{1}{R_{0}%
}\left\Vert \Delta v_{l}\right\Vert _{2}^{2}+\frac{c}{R_{0}}\left\Vert
v_{l}\right\Vert _{W^{1,2}}^{2}\right) \\
&  \leq\left(  1-3\tau\left\Vert \Delta v_{l}\right\Vert _{2}^{2}-3\tau
^{2}\left\Vert \nabla v_{l}\right\Vert _{2}^{2}-\tau^{3}\left\Vert
v_{l}\right\Vert _{2}^{2}\right)  ^{1/2}\left(  1+c\frac{1}{R_{0}}\left\Vert
\Delta v_{l}\right\Vert _{2}^{2}+\frac{c}{R_{0}}\left\Vert v_{l}\right\Vert
_{W^{1,2}}^{2}\right) \\
&  \leq\left(  1-\frac{3\tau}{2}\left\Vert \Delta v_{l}\right\Vert _{2}%
^{2}-\frac{3\tau^{2}}{2}\left\Vert \nabla v_{l}\right\Vert _{2}^{2}-\frac
{\tau^{3}}{2}\left\Vert v_{l}\right\Vert _{2}^{2}\right)  \left(  1+c\frac
{1}{R_{0}}\left\Vert \Delta v_{l}\right\Vert _{2}^{2}+\frac{c}{R_{0}%
}\left\Vert v_{l}\right\Vert _{W^{1,2}}^{2}\right) \\
&  \leq1
\end{align*}
if we choose $R_{0}$ sufficiently large. Furthermore,
\[
\int_{B_{R_{0}}}\left(  e^{\beta_{0}v_{l}^{2}}-1\right)  dx\leq C(R_{0}%
)\int_{B_{R_{0}}}e^{\beta_{0}w_{l}^{2}}dx.
\]
Thus by Theorem B, we have that $I_{1}$ is bounded by a constant depending
only on $\tau.$

Now, we will estimate $I_{2}$. We choose $R_{0}\geq\left[  \frac{1}%
{3\sigma_{6}}\left(  \frac{1}{\tau^{3}}+\frac{1}{3\tau^{2}}\right)  \right]  $
$^{\frac{1}{5}}$ then from the Radial lemma 2.4, we get that $\left\vert
v_{l}(x)\right\vert \leq1$ when $\left\vert x\right\vert \geq R_{0}$. Thus we have%

\begin{align*}
I_{2}  &  =\int_{B_{R_{l}}\smallsetminus B_{R_{0}}}\left(  e^{\beta_{0}%
v_{l}^{2}}-1\right)  dx\\
&  \leq%
{\displaystyle\sum\limits_{j=1}^{\infty}}
\frac{\beta_{0}^{j}}{j!}\int_{B_{R_{l}}}v_{l}^{2}\\
&  \leq\frac{1}{\tau^{3}}%
{\displaystyle\sum\limits_{j=1}^{\infty}}
\frac{\beta_{0}^{j}}{j!}.
\end{align*}
Thus we have that $\int_{B_{R_{l}}}\left(  e^{\beta_{0}v_{l}^{2}}-1\right)
dx$ is bounded by a constant depending only on $\tau$.

Combining the above estimates and using Fatou's lemma, we can conclude that
\[
\underset{\left\Vert \nabla\left(  -\Delta+\tau I\right)  u\right\Vert
_{2}^{2}+\tau\left\Vert \left(  -\Delta+\tau I\right)  u\right\Vert _{2}%
^{2}\leq1}{\sup}\int_{%
\mathbb{R}
^{6}}\left(  e^{\beta_{0}u^{2}}-1\right)  dx<\infty\text{.}%
\]
This completes the proof of Theorem 1.2.
\end{proof}

\textbf{Proofs of Theorem 1.3 and Theorem 1.4 when }$n=2m=6$\textbf{:}

To prove Theorem 1.3 when $m$ is odd, it suffices to choose $\tau>0$ as in
Lemma 2.2. Then we have
\[
\int_{%
\mathbb{R}
^{6}}\left(
{\displaystyle\sum\limits_{j=0}^{3}}
a_{m-j}\left\vert \nabla^{j}u\right\vert ^{2}\right)  dx\geq\left\Vert
\nabla\left(  -\Delta+\tau I\right)  u\right\Vert _{2}^{2}+\tau\left\Vert
\left(  -\Delta+\tau I\right)  u\right\Vert _{2}^{2}%
\]
and we get%
\[
\underset{\int_{%
\mathbb{R}
^{6}}\left(
{\displaystyle\sum\limits_{j=0}^{3}}
a_{m-j}\left\vert \nabla^{j}u\right\vert ^{2}\right)  dx\leq1}{\sup}\int_{%
\mathbb{R}
^{6}}\left(  e^{\beta_{0}u^{2}}-1\right)  dx\leq\underset{\left\Vert
\nabla\left(  -\Delta+\tau I\right)  u\right\Vert _{2}^{2}+\tau\left\Vert
\left(  -\Delta+\tau I\right)  u\right\Vert _{2}^{2}\leq1}{\sup}\int_{%
\mathbb{R}
^{6}}\left(  e^{\beta_{0}u^{2}}-1\right)  dx
\]
When $\beta>\beta_{0}$, it is showed by Kozono, Sato and Wadade \cite{KSW} and
Proposition 6.2 in \cite{RS} that the supremum in Theorem 1.3 is infinite. In
fact, the sequence of test functions which gives the sharpness of Adams'
inequality in bounded domains in \cite{A} gives also the sharpness of Adams'
inequality in unbounded domains. This completes the proof of Theorem 1.3.

Moreover, we can choose $a_{0}=$ $a_{1}=$ $a_{2}=$ $a_{3}=1$ to get Theorem 1.4.

\subsection{Proof of Theorem 1.2 when $m=2k+1,~k\in%
\mathbb{N}
$}

The idea to prove the Adams type inequality in this case is a combination of
ideas in the previous subsection and ideas in Section 4.

\begin{proof}
Let $u\in W^{m,2}\left(
\mathbb{R}
^{2m}\right)  $ be such that
\[
\left\Vert \nabla\left(  -\Delta+\tau I\right)  ^{k}u\right\Vert _{2}^{2}%
+\tau\left\Vert \left(  -\Delta+\tau I\right)  ^{k}u\right\Vert _{2}^{2}\leq1.
\]
By density arguments, we can find a sequence of functions $u_{l}\in
C_{0}^{\infty}\left(
\mathbb{R}
^{2m}\right)  $ such that $u_{l}\rightarrow u$ in $W^{m,2}\left(
\mathbb{R}
^{2m}\right)  $, $\int_{%
\mathbb{R}
^{2m}}\left(  \left\vert \nabla\left(  -\Delta+\tau I\right)  ^{k}%
u_{l}\right\vert ^{2}+\tau\left\vert \left(  -\Delta+\tau I\right)  ^{k}%
u_{l}\right\vert ^{2}\right)  dx\leq1$ and supp$u_{l}\subset B_{R_{l}}$ for
any fixed $l$. Let $f_{l}:=\left(  -\Delta+\tau I\right)  ^{k}u_{l}$. Consider
the problem%
\[
\left\{
\begin{array}
[c]{c}%
\left(  -\Delta+\tau I\right)  ^{k}v_{l}=f_{l}^{\#}\\
v_{l}\in W_{N}^{m,2}\left(  B_{R_{l}}\right)
\end{array}
\right.  .
\]
Such a $v_{l}$ does exist by Section 3 and Remark 3.1. Moreover, by the
properties of rearrangement, we have
\begin{align*}
\int_{B_{R_{l}}}\left\vert \left(  -\Delta+\tau I\right)  ^{k}v_{l}\right\vert
^{2}dx &  =\int_{B_{R_{l}}}\left\vert \left(  -\Delta+\tau I\right)  ^{k}%
u_{l}\right\vert ^{2}dx\\
\int_{B_{R_{l}}}\left\vert \nabla\left(  -\Delta+\tau I\right)  ^{k}%
v_{l}\right\vert ^{2}dx &  \leq\int_{B_{R_{l}}}\left\vert \nabla\left(
-\Delta+\tau I\right)  ^{k}u_{l}\right\vert ^{2}dx
\end{align*}
which leads to
\begin{equation}
\int_{%
\mathbb{R}
^{2m}}\left(  \left\vert \nabla\left(  -\Delta+\tau I\right)  ^{k}%
v_{l}\right\vert ^{2}+\tau\left\vert \left(  -\Delta+\tau I\right)  ^{k}%
v_{l}\right\vert ^{2}\right)  dx\leq1\label{8.1}%
\end{equation}
Note that from (\ref{8.1}) and the formula (\ref{2.07}), we have%
\begin{align*}
\left\Vert v_{l}\right\Vert _{W^{1,2}} &  =\left[  \int_{B_{R_{l}}}\left(
\left\vert v_{l}\right\vert ^{2}+\left\vert \nabla v_{l}\right\vert
^{2}\right)  \right]  ^{1/2}\\
&  \leq\sqrt{\frac{1}{\tau^{m}}+\frac{1}{m\tau^{m-1}}}.
\end{align*}
By Corollary 3.1, we get
\[
\int_{B_{R_{l}}}\left(  e^{\beta_{0}u_{l}^{2}}-1\right)  dx=\int_{B_{R_{l}}%
}\left(  e^{\beta_{0}u_{l}^{\#2}}-1\right)  dx\leq\int_{B_{R_{l}}}\left(
e^{\beta_{0}v_{l}^{2}}-1\right)  dx.
\]
Here, $\beta_{0}=\beta(2m,m)$.

Again, we write
\begin{align*}
\int_{B_{R_{l}}}\left(  e^{\beta_{0}v_{l}^{2}}-1\right)  dx &  \leq
\int_{B_{R_{0}}}\left(  e^{\beta_{0}v_{l}^{2}}-1\right)  dx+\int_{B_{R_{l}%
}\smallsetminus B_{R_{0}}}\left(  e^{\beta_{0}v_{l}^{2}}-1\right)  dx\\
&  =I_{1}+I_{2}%
\end{align*}
where $R_{0}$ depends only on $\tau$ and will be chosen later. We will prove
that both $I_{1}$ and $I_{2}$ are bounded uniformly by a constant that depends
only on $\tau$.

First, we will estimate $I_{2}$. We choose $R_{0}\geq\left[  \frac{1}%
{m\sigma_{n}}\left(  \frac{1}{\tau^{m}}+\frac{1}{m\tau^{m-1}}\right)  \right]
$ $^{\frac{1}{n-1}}$ then from the Radial lemma 2.4, we get that $\left\vert
v_{l}(x)\right\vert \leq1$ when $\left\vert x\right\vert \geq R_{0}$. Thus we
have%
\begin{align*}
I_{2}  &  =\int_{B_{R_{l}}\smallsetminus B_{R_{0}}}\left(  e^{\beta_{0}%
v_{l}^{2}}-1\right)  dx\\
&  \leq%
{\displaystyle\sum\limits_{j=1}^{\infty}}
\frac{\beta_{0}^{j}}{j!}\int_{B_{R_{l}}}v_{l}^{2}\\
&  \leq\frac{1}{\tau^{m}}%
{\displaystyle\sum\limits_{j=1}^{\infty}}
\frac{\beta_{0}^{j}}{j!}.
\end{align*}
Thus we have that $\int_{B_{R_{l}}}\left(  e^{\beta_{0}v_{l}^{2}}-1\right)
dx$ is bounded by a constant depending only on $\tau$.

To estimate $I_{1}$, again, we need to construct an auxiliary radial function
$w_{l}\in W_{N}^{m,2}\left(  B_{R_{0}}\right)  $ with $\left\Vert \nabla
^{m}w_{l}\right\Vert _{2}\leq1$ which increases the integral we are interested
in. We will construct such the function by the very similar way as in the case
$m$ is even \cite{RS} and the case $m=3$. For each $i\in\left\{
0,1,2,...k-1\right\}  $ we define
\[
g_{i}\left(  \left\vert x\right\vert \right)  :=\left\vert x\right\vert
^{m-1-2i},~\forall x\in B_{R_{0}}%
\]
so $g_{i}\in W_{rad}^{m,2}\left(  B_{R_{0}}\right)  $. Moreover,%
\[
\Delta^{j}g_{i}\left(  \left\vert x\right\vert \right)  =\left\{
\begin{array}
[c]{c}%
c_{i}^{j}\left\vert x\right\vert ^{m-1-2(i+j)}\text{ for }j\in\left\{
1,2,...k-i\right\} \\
0\text{ for }j\in\left\{  k-i+1,...,k\right\}
\end{array}
\right.  ~\forall x\in B_{R_{0}}%
\]
where
\[
c_{i}^{j}=%
{\displaystyle\prod\limits_{h=1}^{j}}
\left[  6k-2\left(  i+h-1\right)  \right]  \left[  2k-2\left(  i+h-1\right)
\right]  \text{, }\forall j\in\left\{  1,2,...k-i\right\}  .
\]
Let
\[
z_{l}\left(  \left\vert x\right\vert \right)  :=v_{l}\left(  \left\vert
x\right\vert \right)  -%
{\displaystyle\sum\limits_{i=0}^{k-1}}
a_{i}g_{i}\left(  \left\vert x\right\vert \right)  -a_{k},~\forall x\in
B_{R_{0}}%
\]
where
\begin{align*}
a_{0}  &  :=\frac{\Delta^{k}v_{l}\left(  R_{0}\right)  }{\Delta^{k}g\left(
R_{0}\right)  }\\
a_{i}  &  :=\frac{\Delta^{k-i}v_{l}\left(  R_{0}\right)  -%
{\displaystyle\sum\limits_{j=0}^{i-1}}
a_{j}\Delta^{k-i}g_{j}\left(  R_{0}\right)  }{\Delta^{k-i}g_{i}\left(
R_{0}\right)  },~\forall i\in\left\{  1,2,...k-1\right\}  ,\\
a_{k}  &  :=v_{l}\left(  R_{0}\right)  -%
{\displaystyle\sum\limits_{i=0}^{k-1}}
a_{i}g_{i}\left(  R_{0}\right)  .
\end{align*}
We can check that%
\begin{align*}
z_{l}  &  \in W_{N,rad}^{m,2}\left(  B_{R_{0}}\right)  ,\\
\nabla^{m}v_{l}  &  =\nabla^{m}z_{l}\text{ in }B_{R_{0}}\text{.}%
\end{align*}
Combining the proofs when $m$ is even in \cite{RS} and when $m=3$, the Radial
Lemma 2.4 and 2.5, we have for $R_{0}\geq1$

\begin{lemma}
For $0<\left\vert x\right\vert \leq R_{0}$, there exists some positive constant $d(m, R_0)$ depending only on $m$ and $R_0$ such that%
\begin{align*}
\left\vert v_{l}\left(  \left\vert x\right\vert \right)  \right\vert ^{2}  &
\leq\left\vert z_{l}\left(  \left\vert x\right\vert \right)  \right\vert
^{2}\left(  1+c_{m}\frac{1}{R_{0}}\left\Vert \Delta^{k}v_{l}\right\Vert
_{2}^{2}+c_{m}%
{\displaystyle\sum\limits_{j=1}^{k-1}}
\frac{1}{R_{0}}\left\Vert \Delta^{k-j}v_{l}\right\Vert _{W^{1,2}}^{2}%
+\frac{c_{m}}{R_{0}}\left\Vert v_{l}\right\Vert _{W^{1,2}}^{2}\right)  ^{2}\\
& +d(m,R_{0}).
\end{align*}

\end{lemma}

Now, setting
\[
w_{l}\left(  \left\vert x\right\vert \right)  :=z_{l}\left(  \left\vert
x\right\vert \right)  \left(  1+c_{m}\frac{1}{R_{0}}\left\Vert \Delta^{k}%
v_{l}\right\Vert _{2}^{2}+c_{m}%
{\displaystyle\sum\limits_{j=0}^{k-1}}
\frac{1}{R_{0}}\left\Vert \Delta^{k-j}v_{l}\right\Vert _{W^{1,2}}^{2}%
+\frac{c_{m}}{R_{0}}\left\Vert v_{l}\right\Vert _{W^{1,2}}^{2}\right)  .
\]
Since
\begin{align*}
z_{l}  &  \in W_{N,rad}^{m,2}\left(  B_{R_{0}}\right)  ,\\
\nabla^{m}v_{l}  &  =\nabla^{m}z_{l}\text{ in }B_{R_{0}}\text{.}%
\end{align*}
we have
\[
w_{l}\in W_{N,rad}^{m,2}\left(  B_{R_{0}}\right)
\]
and%
\[
\left\Vert \nabla^{m}w_{l}\right\Vert _{2}=\left\Vert \nabla^{m}%
z_{l}\right\Vert _{2}\left(  1+c_{m}\frac{1}{R_{0}}\left\Vert \Delta^{k}%
v_{l}\right\Vert _{2}^{2}+c_{m}%
{\displaystyle\sum\limits_{j=1}^{k-1}}
\frac{1}{R_{0}}\left\Vert \Delta^{k-j}v_{l}\right\Vert _{W^{1,2}}^{2}%
+\frac{c_{m}}{R_{0}}\left\Vert v_{l}\right\Vert _{W^{1,2}}^{2}\right)  .
\]
Note that
\begin{align*}
\left\Vert \nabla^{m}z_{l}\right\Vert _{2}  &  =\left\Vert \nabla^{m}%
v_{l}\right\Vert _{2}\\
&  \leq\left(  1-\lambda\left\Vert \Delta^{k}v_{l}\right\Vert _{2}^{2}-\lambda%
{\displaystyle\sum\limits_{j=1}^{k-1}}
\left\Vert \Delta^{k-j}v_{l}\right\Vert _{W^{1,2}}^{2}-\lambda\left\Vert
v_{l}\right\Vert _{W^{1,2}}^{2}\right)  ^{1/2}\\
&  \leq1-\frac{\lambda}{2}\left\Vert \Delta^{k}v_{l}\right\Vert _{2}^{2}%
-\frac{\lambda}{2}%
{\displaystyle\sum\limits_{j=1}^{k-1}}
\left\Vert \Delta^{k-j}v_{l}\right\Vert _{W^{1,2}}^{2}-\frac{\lambda}%
{2}\left\Vert v_{l}\right\Vert _{W^{1,2}}^{2}%
\end{align*}
where
\[
\lambda=\min\left\{  \binom{m}{j}\tau^{m-j}:j=0,1,...,m.\right\}
\]
we have%
\begin{align*}
\left\Vert \nabla^{m}w_{l}\right\Vert _{2}  &  \leq\left(  1-\frac{\lambda}%
{2}\left\Vert \Delta^{k}v_{l}\right\Vert _{2}^{2}-\frac{\lambda}{2}%
{\displaystyle\sum\limits_{j=1}^{k-1}}
\left\Vert \Delta^{k-j}v_{l}\right\Vert _{W^{1,2}}^{2}-\frac{\lambda}%
{2}\left\Vert v_{l}\right\Vert _{W^{1,2}}^{2}\right)  \times\\
&  \times\left(  1+c_{m}\frac{1}{R_{0}}\left\Vert \Delta^{k}v_{l}\right\Vert
_{2}^{2}+c_{m}%
{\displaystyle\sum\limits_{j=1}^{k-1}}
\frac{1}{R_{0}}\left\Vert \Delta^{k-j}v_{l}\right\Vert _{W^{1,2}}^{2}%
+\frac{c_{m}}{R_{0}}\left\Vert v_{l}\right\Vert _{W^{1,2}}^{2}\right) \\
&  \leq1
\end{align*}
if we choose $R_{0}=R_{0}(\tau)$ sufficiently large.

Finally, note that
\[
I_{1}\leq e^{\beta_{0}d(m,R_{0})}\int_{B_{R_{0}}}e^{\beta_{0}w_{l}^{2}}dx,
\]
by using Theorem B, we can conclude that $I_{1}$ is bounded by a constant
depending only on $\tau$ since $\left\Vert \nabla^{m}w_{l}\right\Vert _{2}%
\leq1.$

Combining the above estimates and applying Fatou's lemma, we can conclude
that
\[
\underset{u\in W^{m,2}\left(
\mathbb{R}
^{2m}\right)  ,\left\Vert \nabla\left(  -\Delta+\tau I\right)  ^{k}%
u\right\Vert _{2}^{2}+\tau\left\Vert \left(  -\Delta+\tau I\right)
^{k}u\right\Vert _{2}^{2}\leq1}{\sup}\int_{%
\mathbb{R}
^{2m}}\left(  e^{\beta_{0}u^{2}}-1\right)  dx<\infty.
\]

\end{proof}

\textbf{Proofs of Theorem 1.3 and Theorem 1.4 when }$m$ is odd\textbf{: }From
Lemma 2.2, we have the conclusion of Theorem 1.3 when $m=2k+1,~k\in%
\mathbb{N}
$.

Again, when $\beta>\beta_{0}$, we can check that the sequence of test
functions which gives the sharpness of Adams' inequality in bounded domains in
\cite{A} gives also the sharpness of Adams' inequalities in unbounded domains.
See Proposition 6.2 in \cite{RS}.

Moreover, we can choose $a_{j}=1,~j=0,...,m$ to get the Theorem 1.4.

\section{Proof of Theorem 1.1}

\begin{proof}
Let $u\in W^{m,\frac{n}{m}}\left(
\mathbb{R}
^{n}\right)  $ be such that
\[
\left\Vert \nabla\left(  -\Delta+I\right)  ^{k}u\right\Vert _{\frac{n}{m}%
}^{\frac{n}{m}}+\left\Vert \left(  -\Delta+I\right)  ^{k}u\right\Vert
_{\frac{n}{m}}^{\frac{n}{m}}\leq1,
\]
by density arguments, we can find a sequence of functions $u_{l}\in
C_{0}^{\infty}\left(
\mathbb{R}
^{n}\right)  $ such that $u_{l}\rightarrow u$ in $W^{m,\frac{n}{m}}\left(
\mathbb{R}
^{n}\right)  $, $\int_{%
\mathbb{R}
^{n}}\left(  \left\vert \nabla\left(  -\Delta+I\right)  ^{k}u_{l}\right\vert
^{\frac{n}{m}}+\left\vert \left(  -\Delta+I\right)  ^{k}u_{l}\right\vert
^{\frac{n}{m}}\right)  dx\leq1$ and supp\thinspace\ $u_{l}\subset B_{R_{l}}$
for any fixed $l$. Let $f_{l}:=\left(  -\Delta+I\right)  ^{k}u_{l}$ and
consider the problem%
\[
\left\{
\begin{array}
[c]{c}%
\left(  -\Delta+I\right)  ^{k}v_{l}=f_{l}^{\#}\\
v_{l}\in W_{N}^{m,2}\left(  B_{R_{l}}\right)
\end{array}
\right.  .
\]
By the properties of rearrangement, we have
\begin{align*}
\int_{B_{R_{l}}}\left\vert \left(  -\Delta+I\right)  ^{k}v_{l}\right\vert
^{\frac{n}{m}}dx &  =\int_{B_{R_{l}}}\left\vert \left(  -\Delta+I\right)
^{k}u_{l}\right\vert ^{\frac{n}{m}}dx\\
\int_{B_{R_{l}}}\left\vert \nabla\left(  -\Delta+I\right)  ^{k}v_{l}%
\right\vert ^{\frac{n}{m}}dx &  \leq\int_{B_{R_{l}}}\left\vert \nabla\left(
-\Delta+I\right)  ^{k}u_{l}\right\vert ^{\frac{n}{m}}dx
\end{align*}
Therefore, we have
\begin{equation}
\int_{%
\mathbb{R}
^{n}}\left(  \left\vert \nabla\left(  -\Delta+I\right)  ^{k}v_{l}\right\vert
^{\frac{n}{m}}+\left\vert \left(  -\Delta+I\right)  ^{k}v_{l}\right\vert
^{\frac{n}{m}}\right)  dx\leq1\label{9.1}%
\end{equation}
By Corollary 2.1, we get
\[
\int_{B_{R_{l}}}\phi\left(  \beta_{0}\left\vert u_{l}\right\vert ^{\frac
{n}{n-m}}\right)  dx=\int_{B_{R_{l}}}\phi\left(  \beta_{0}\left\vert
u_{l}^{\#}\right\vert ^{\frac{n}{n-m}}\right)  dx\leq\int_{B_{R_{l}}}%
\phi\left(  \beta_{0}\left\vert v_{l}\right\vert ^{\frac{n}{n-m}}\right)  dx
\]
Here, $\beta_{0}=\beta(n,m)$.

Again, write
\begin{align*}
\int_{B_{R_{l}}}\phi\left(  \beta_{0}\left\vert v_{l}\right\vert ^{\frac
{n}{n-m}}\right)  dx &  \leq\int_{B_{R_{0}}}\phi\left(  \beta_{0}\left\vert
v_{l}\right\vert ^{\frac{n}{n-m}}\right)  dx+\int_{B_{R_{l}}\smallsetminus
B_{R_{0}}}\phi\left(  \beta_{0}\left\vert v_{l}\right\vert ^{\frac{n}{n-m}%
}\right)  dx\\
&  =I_{1}+I_{2}%
\end{align*}
where $R_{0}$ is a positive constant and will be chosen later. We will prove
that both $I_{1}$ and $I_{2}$ are bounded uniformly.

To do that, again, first, we need to construct an auxiliary radial function
$w_{l}\in W_{N}^{m,\frac{n}{m}}\left(  B_{R_{0}}\right)  $ with $\left\Vert
\nabla^{m}w_{l}\right\Vert _{\frac{n}{m}}\leq1$ which increases the integral
$I_{1}$. For each $i\in\left\{  0,1,2,...k-1\right\}  $ we define
\[
g_{i}\left(  \left\vert x\right\vert \right)  :=\left\vert x\right\vert
^{m-1-2i},~\forall x\in B_{R_{0}}%
\]
so $g_{i}\in W_{rad}^{m,\frac{n}{m}}\left(  B_{R_{0}}\right)  $. Moreover,%
\[
\Delta^{j}g_{i}\left(  \left\vert x\right\vert \right)  =\left\{
\begin{array}
[c]{c}%
c_{i}^{j}\left\vert x\right\vert ^{m-1-2(i+j)}\text{ for }j\in\left\{
1,2,...k-i\right\} \\
0\text{ for }j\in\left\{  k-i+1,...,k\right\}
\end{array}
\right.  ~\forall x\in B_{R_{0}}%
\]
where
\[
c_{i}^{j}=%
{\displaystyle\prod\limits_{h=1}^{j}}
\left[  6k-2\left(  i+h-1\right)  \right]  \left[  2k-2\left(  i+h-1\right)
\right]  \text{, }\forall j\in\left\{  1,2,...k-i\right\}  .
\]
Let
\[
z_{l}\left(  \left\vert x\right\vert \right)  :=v_{l}\left(  \left\vert
x\right\vert \right)  -%
{\displaystyle\sum\limits_{i=0}^{k-1}}
a_{i}g_{i}\left(  \left\vert x\right\vert \right)  -a_{k},~\forall x\in
B_{R_{0}}%
\]
where
\begin{align*}
a_{0}  &  :=\frac{\Delta^{k}v_{l}\left(  R_{0}\right)  }{\Delta^{k}g\left(
R_{0}\right)  }\\
a_{i}  &  :=\frac{\Delta^{k-i}v_{l}\left(  R_{0}\right)  -%
{\displaystyle\sum\limits_{j=0}^{i-1}}
a_{j}\Delta^{k-i}g_{j}\left(  R_{0}\right)  }{\Delta^{k-i}g_{i}\left(
R_{0}\right)  },~\forall i\in\left\{  1,2,...k-1\right\}  ,\\
a_{k}  &  :=v_{l}\left(  R_{0}\right)  -%
{\displaystyle\sum\limits_{i=0}^{k-1}}
a_{i}g_{i}\left(  R_{0}\right)  .
\end{align*}
We can check that%
\begin{align*}
z_{l}  &  \in W_{N,rad}^{m,\frac{n}{m}}\left(  B_{R_{0}}\right)  ,\\
\nabla^{m}v_{l}  &  =\nabla^{m}z_{l}\text{ in }B_{R_{0}}\text{.}%
\end{align*}
By a similar argument to that in \cite{RS}, and a combination of Radial Lemmas
2.4 and 2.5, we can prove that for $R_{0}\geq1$

\begin{lemma}
For $0<\left\vert x\right\vert \leq R_{0}$ we have for some constant
$d(m,n,R_{0})$ such that
\begin{align*}
&  \left\vert v_{l}\left(  \left\vert x\right\vert \right)  \right\vert
^{\frac{n}{m}}\\
\leq &  \left\vert z_{l}\left(  \left\vert x\right\vert \right)  \right\vert
^{\frac{n}{m}}\left(  1+c_{m,n}\frac{1}{R_{0}}\left\Vert \Delta^{k}%
v_{l}\right\Vert _{\frac{n}{m}}^{\frac{n}{m}}+c_{m,n}%
{\displaystyle\sum\limits_{j=1}^{k-1}}
\frac{1}{R_{0}}\left\Vert \Delta^{k-j}v_{l}\right\Vert _{W^{1,\frac{n}{m}}%
}^{\frac{n}{m}}+\frac{c_{m,n}}{R_{0}}\left\Vert v_{l}\right\Vert
_{W^{1,\frac{n}{m}}}^{\frac{n}{m}}\right)  ^{\frac{n}{m}}\\
&  +d(m,n,R_{0}).
\end{align*}

\end{lemma}

Now, setting
\[
w_{l}\left(  \left\vert x\right\vert \right)  :=z_{l}\left(  \left\vert
x\right\vert \right)  \left(  1+c_{m,n}\frac{1}{R_{0}}\left\Vert \Delta
^{k}v_{l}\right\Vert _{\frac{n}{m}}^{\frac{n}{m}}+c_{m,n}%
{\displaystyle\sum\limits_{j=1}^{k-1}}
\frac{1}{R_{0}}\left\Vert \Delta^{k-j}v_{l}\right\Vert _{W^{1,\frac{n}{m}}%
}^{\frac{n}{m}}+\frac{c_{m,n}}{R_{0}}\left\Vert v_{l}\right\Vert
_{W^{1,\frac{n}{m}}}^{\frac{n}{m}}\right)  .
\]
Since
\begin{align*}
z_{l} &  \in W_{N,rad}^{m,\frac{n}{m}}\left(  B_{R_{0}}\right)  ,\\
\nabla^{m}v_{l} &  =\nabla^{m}z_{l}\text{ in }B_{R_{0}}\text{.}%
\end{align*}
we have
\[
w_{l}\in W_{N,rad}^{m,\frac{n}{m}}\left(  B_{R_{0}}\right)
\]
and%
\[
\left\Vert \nabla^{m}w_{l}\right\Vert _{\frac{n}{m}}=\left\Vert \nabla
^{m}z_{l}\right\Vert _{\frac{n}{m}}\left(  1+\frac{c_{m,n}}{R_{0}}\left\Vert
\Delta^{k}v_{l}\right\Vert _{\frac{n}{m}}^{\frac{n}{m}}+%
{\displaystyle\sum\limits_{j=1}^{k-1}}
\frac{c_{m,n}}{R_{0}}\left\Vert \Delta^{k-j}v_{l}\right\Vert _{W^{1,\frac
{n}{m}}}^{\frac{n}{m}}+\frac{c_{m,n}}{R_{0}}\left\Vert v_{l}\right\Vert
_{W^{1,\frac{n}{m}}}^{\frac{n}{m}}\right)  .
\]
Note that from Lemma 2.3:
\begin{align*}
\left\Vert \nabla^{m}z_{l}\right\Vert _{\frac{n}{m}} &  =\left\Vert \nabla
^{m}v_{l}\right\Vert _{\frac{n}{m}}\\
&  \leq\left(  1-\frac{1}{C}\left\Vert \Delta^{k}v_{l}\right\Vert _{\frac
{n}{m}}^{\frac{n}{m}}-\frac{1}{C}%
{\displaystyle\sum\limits_{j=1}^{k-1}}
\left\Vert \Delta^{k-j}v_{l}\right\Vert _{W^{1,\frac{n}{m}}}^{\frac{n}{m}%
}-\frac{1}{C}\left\Vert v_{l}\right\Vert _{W^{1,\frac{n}{m}}}^{\frac{n}{m}%
}\right)  ^{m/n}\\
&  \leq1-\frac{m}{nC}\left\Vert \Delta^{k}v_{l}\right\Vert _{\frac{n}{m}%
}^{\frac{n}{m}}-\frac{m}{nC}%
{\displaystyle\sum\limits_{j=1}^{k-1}}
\left\Vert \Delta^{k-j}v_{l}\right\Vert _{W^{1,\frac{n}{m}}}^{\frac{n}{m}%
}-\frac{m}{nC}\left\Vert v_{l}\right\Vert _{W^{1,\frac{n}{m}}}^{\frac{n}{m}},
\end{align*}
we have%
\begin{align*}
\left\Vert \nabla^{m}w_{l}\right\Vert _{\frac{n}{m}} &  \leq\left(  1-\frac
{m}{nC}\left\Vert \Delta^{k}v_{l}\right\Vert _{\frac{n}{m}}^{\frac{n}{m}%
}-\frac{m}{nC}%
{\displaystyle\sum\limits_{j=1}^{k-1}}
\left\Vert \Delta^{k-j}v_{l}\right\Vert _{W^{1,\frac{n}{m}}}^{\frac{n}{m}%
}-\frac{m}{nC}\left\Vert v_{l}\right\Vert _{W^{1,\frac{n}{m}}}^{\frac{n}{m}%
}\right)  \times\\
&  \times\left(  1+c_{m,n}\frac{1}{R_{0}}\left\Vert \Delta^{k}v_{l}\right\Vert
_{\frac{n}{m}}^{\frac{n}{m}}+c_{m,n}%
{\displaystyle\sum\limits_{j=1}^{k-1}}
\frac{1}{R_{0}}\left\Vert \Delta^{k-j}v_{l}\right\Vert _{W^{1,\frac{n}{m}}%
}^{\frac{n}{m}}+\frac{c_{m,n}}{R_{0}}\left\Vert v_{l}\right\Vert
_{W^{1,\frac{n}{m}}}^{\frac{n}{m}}\right)  \\
&  \leq1
\end{align*}
if we choose $R_{0}$ sufficiently large.

Finally, note that
\[
I_{1}\leq e^{\beta_{0}d(m,n,R_{0})}\int_{B_{R_{0}}}e^{\beta_{0}w_{l}^{2}}dx,
\]
by using Theorem B, we can conclude that $I_{1}$ is bounded by a constant
depending only on $n$ and $m$.

Now, by the same argument as in \cite{RS} and noting that from (\ref{9.1}) and
Lemma 2.3, we have $\left\Vert v_{l}\right\Vert _{W^{1,\frac{n}{m}}}\leq D$
for some constant $D>0,$ we can conclude that $I_{2}$ is also bounded by a
constant depending only on $n$ and $m$.

Combining the above estimates and employing Fatou's lemma, we can conclude
that
\[
\underset{u\in W^{m,\frac{n}{m}}\left(
\mathbb{R}
^{n}\right)  ,\left\Vert \nabla\left(  -\Delta+I\right)  ^{k}u\right\Vert
_{\frac{n}{m}}^{\frac{n}{m}}+\left\Vert \left(  -\Delta+I\right)
^{k}u\right\Vert _{\frac{n}{m}}^{\frac{n}{m}}\leq1}{\sup}\int_{%
\mathbb{R}
^{n}}\phi\left(  \beta\left(  n,m\right)  \left\vert u\right\vert ^{\frac
{n}{n-m}}\right)  dx<\infty.
\]
Again, when $\beta>\beta\left(  n,m\right)  $, it is showed by Kozono, Sato
and Wadade \cite{KSW} that the supremum is infinite. In fact, the sequence of
test functions which gives the sharpness of Adams' inequality in bounded
domains in \cite{A} gives also the sharpness of Adams' inequality in unbounded
domains (see Proposition 6.2 in \cite{RS}).
\end{proof}

\textbf{Acknowledgement:} The results of this paper were presented at the
International Conference in Geometry, Analysis and PDEs  at Jiaxing University in
China. The authors wish to thank the local organizers for invitation and hospitality.

\end{document}